\documentclass[11pt,reqno]{amsart}

\usepackage{amsmath,amssymb,amsthm}
\usepackage{tikz}
\usepackage{graphicx, subcaption}
\usepackage{hyperref}
\usepackage[margin=3cm]{geometry}
\usepackage{stmaryrd} 

\usepackage{genyoungtabtikz}
\YFrench
\Yboxdim{0.5cm}

\colorlet{myblue}{cyan!80!blue}
\colorlet{myred}{red!50}

\newcommand{\Yred}{\Yfillcolour{myred}}
\newcommand{\Ycyan}{\Yfillcolour{myblue}}
\newcommand{\Ygray}{\Yfillcolour{lightgray}}

\newtheorem{thm}{Theorem}[section]
\newtheorem{lem}[thm]{Lemma}

\newtheorem{prop}[thm]{Proposition}
\newtheorem{rem}[thm]{Remark}
\newtheorem{ex}[thm]{Example}
\newtheorem{defi}[thm]{Definition}

\newcommand{\la}{\lambda}
\newcommand{\x}{\mathbf{x}}
\newcommand{\y}{\mathbf{y}}
\newcommand{\lb}{\llbracket}
\newcommand{\rb}{\rrbracket}

\DeclareMathOperator{\F}{F}
\newcommand{\GL}{{\rm GL}}
\newcommand{\SSYT}{{\rm SSYT}}
\newcommand{\RSK}{{\rm RSK}}
\newcommand{\C}{\mathbb{C}}
\newcommand{\col}{{\rm col}}
\newcommand{\row}{{\rm row}}
\newcommand{\asym}{{\rm asym}}
\newcommand{\odd}{{\rm odd}}
\newcommand{\rk}{{\rm rk}}

\newcommand{\Ret}{\mathbf{R}}
\newcommand{\Set}{\mathbf{S}}

\newcommand{\deff}[1]{\emph{\color{cyan!80!blue} #1}} 
\newcommand{\bb}[1]{\mathbb{#1}}
\newcommand{\mc}[1]{\mathcal{#1}}
\newcommand{\wh}[1]{\widehat{#1}}

\newcommand{\YoungBlue}[2]{
\begin{tikzpicture}
\tyng(0cm,0cm,#1)
\Ycyan
\tgyoung(0cm,0cm,#2)
\end{tikzpicture}
}
\newcommand{\YoungRed}[2]{
\begin{tikzpicture}
\tyng(0cm,0cm,#1)
\Yred
\tgyoung(0cm,0cm,#2)
\end{tikzpicture}
}

\title{Growth diagram proofs for the Littlewood identities}
\author{Florian Schreier-Aigner}
\address[F.~Schreier-Aigner]{University of Vienna, Austria}
\email{florian.schreier-aigner@univie.ac.at}
\thanks{This research was funded in whole by the Austrian Science Fund (FWF) \href{https://dx.doi.org/10.55776/P34931}{doi:10.55776/P34931} and \href{https://dx.doi.org/10.55776/P36863}{doi:10.55776/P36863}.}

\begin{document}

\begin{abstract}
The (dual) Cauchy identity has an easy algebraic proof utilising a commutation relation between the up and (dual) down operators. By using Fomin's growth diagrams, a bijective proof of the commutation relation can be ``bijectivised'' to obtain RSK like correspondences.
In this paper we give a concise overview of this machinery and extend it to Littlewood type identities by introducing a new family of relations between these operators, called projection identities. Thereby we obtain infinite families of bijections for the Littlewood identities generalising the classical ones. We believe that this approach will be useful for finding bijective proofs for Littlewood type identities in other settings such as for Macdonald polynomials and their specialisations, alternating sign matrices or vertex models.



\end{abstract}

\maketitle


\section{Introduction}

The Robinson--Schensted--Knuth correspondence (RSK) is a combinatorial algorithm with many applications in various areas of mathematics including combinatorics, representation theory, geometry, and probability theory.
The algorithm was first described by Robinson \cite{Robinson38} for permutations and independently  by Schensted \cite{Schensted61} for words; the connection between these algorithms was pointed out by Sch\"utzenberger. 
 In a seminal paper Knuth \cite{Knuth70b} extended the previous correspondence to a bijection between matrices with non-negative entries and pairs of semistandard Young tableaux of the same shape, called RSK.
In particular RSK  gives a combinatorial proof of the Cauchy identity for Schur functions
\begin{equation}
\label{eq:Cauchy0}
\sum_{\la} s_\la(\x) s_\la(\y) = \prod_{i,j}\frac{1}{1-x_iy_j}.
\end{equation}
By restricting to the squarefree monomial $x_1 \ldots x_n y_1 \ldots y_n$ on both sides of \eqref{eq:Cauchy0} we obtain the famous identity
\begin{equation}
\label{eq:squarefree Cauchy identity}
\sum_{\lambda \vdash n} (f_\lambda)^2 = n!,
\end{equation}
where $f_\lambda$ is the number of \emph{standard Young Tableaux} (SYTs) of shape $\lambda$. 
There are two important variations of RSK: ``the''\footnote{Actually Burge \cite{Burge74} introduces four correspondences where usually the bijection introduced in \cite[Section 3]{Burge74} is referred as the Burge correspondence.} Burge correspondence \cite{Burge74} which yields another bijective proof of \eqref{eq:Cauchy0} and dual RSK ($\RSK^*$), introduced by Knuth \cite{Knuth70b}, which gives a bijective proof of the dual Cauchy identity
\begin{equation}
\label{eq:dual Cauchy0}
\sum_{\la} s_\la(\x) s_{\la^\prime}(\y) = \prod_{i,j}(1+x_iy_j).
\end{equation}
The fact that RSK and its variations yield bijective proofs of these identities can be seen as one of the reasons that they have many applications. In particular all of the above identities have a representation theoretic interpretation.
 For $\x=(x_1,\ldots,x_n)$ and $\y=(y_1,\ldots,y_m)$ the right hand side of \eqref{eq:Cauchy0} and \eqref{eq:dual Cauchy0} are the characters of the $\GL_n(\mathbb{C}) \times \GL_m(\mathbb{C})$-module $\bigvee\left(\mathbb{C}^n \otimes \mathbb{C}^m\right)$  and  $\bigwedge\left(\mathbb{C}^n \otimes \mathbb{C}^m\right)$  respectively while the left hand sides are their decompositions into irreducible characters.
 Similarly for \eqref{eq:squarefree Cauchy identity}, the right hand side is the dimension of the group algebra $\C[S_n]$ while the left hand gives a formula for the dimension of its decomposition $\bigoplus_{\la \vdash n}(S^\la)^{f_\la}$ into irreducible representations of $S_n$. 
There exist many variations of the above identities by considering representations with respect to different groups or algebras, or by regarding different families of tableaux respectively. In particular these different settings led to variations as well as generalisations of RSK; see for example \cite{Berele86, ColmenarejoOrellanaSaliolaSchillingZabrocki20, HalversonLewandowski05, Okada91, PatelPatelStokke22, Proctor90b, Sundaram90, Sundaram90b, Terada93, vanLeeuwen05}.

In statistical physics and probability theory the Cauchy identity has another interpretation. A partition $\lambda$ is interpreted as a one dimensional particle configuration using its Maya diagram where the particle configuration corresponding to $\la$ appears with probability proportional to $s_\lambda(\x)s_\lambda(\y)$, see for example \cite{Okounkov01}. 
 As a consequence, the Cauchy identity \eqref{eq:Cauchy0} yields a formula for the partition function of this system. It turns out that RSK can thereby be used to study certain stochastic systems such as TASEP. In the last decades many (probabilistic) generalisations of (dual) RSK and its specialisations were introduced to study further probabilistic models, compare for example with \cite{BufetovMatveev18, FriedenSchreierAigner24plus, MatveevPetrov17} and references within. Typically these generalisations of RSK yield proofs for Cauchy identities of symmetric polynomials or symmetric functions generalising Schur polynomials including $q$-Whittaker polynomials, Hall--Littlewood polynomials or Macdonald polynomials. \medskip

A family of identities closely related to the (dual) Cauchy identity are the Littlewood identities
\begin{align}
\label{eq:littlewood even col}
\sum_{\substack{\la \\ \la^\prime \text{ even}}} s_\la(\x) 
&= \prod_{1 \leq i<j \leq n}\frac{1}{1-x_ix_j},\\
\label{eq:littlewood identity}
\sum_{\la} s_\la(\x) 
&= \prod_{i=1}^n\frac{1}{1-x_i}\prod_{1\leq i<j \leq n}\frac{1}{1-x_ix_j},\\
\label{eq:littlewood even row}
\sum_{\substack{\la \\ \la \text{ even}}} s_\la(\x) 
&= \prod_{i=1}^n\frac{1}{1-x_i^2}\prod_{1\leq i<j \leq n}\frac{1}{1-x_ix_j},\\
\nonumber\\
\label{eq:littlewood 1 asym}
\sum_{\substack{\la \\ \text{ $1$-asymmetric}}} s_\la(\x) 
&=\prod_{1 \leq i<j \leq n}(1+x_ix_j) ,\\
\label{eq:littlewood -1 asym}
\sum_{\substack{\la \\ \text{ $-1$-asymmetric}}} s_\la(\x) 
&=\prod_{i=1}^n(1+x_i^2)\prod_{1 \leq i<j \leq n}(1+x_ix_j). 
\end{align}
Note that \eqref{eq:littlewood identity} is often called ``the'' Littlewood identity. These identities have bijective proofs which are based on RSK, the Burge correspondence and dual RSK. The identity \eqref{eq:littlewood identity} follows immediately by a symmetry property of RSK \cite[Theorem 3]{Knuth70b} or the Burge correspondence respectively, the identities \eqref{eq:littlewood even col} and \eqref{eq:littlewood even row} need a closer analysis of RSK \cite[Theorem 4]{Knuth70b}. The first bijective proofs for \eqref{eq:littlewood 1 asym} and \eqref{eq:littlewood -1 asym} were presented by Burge \cite{Burge74}.
Again there is a connection to representation theory, in the case of the Littlewood identities this is by the Weyl denominator formula.
 \medskip

While RSK is usually described as an insertion algorithm, there are various different descriptions such as Fulton's \emph{matrix-ball construction} \cite{Fulton97} building on  Viennot's shadowing \cite{Viennot77}, or most prominently Fomin's description by using \emph{growth diagrams} \cite{Fomin86, Fomin95}. Fomin's construction builds on two key observations. Firstly, we can interpret both sides of the Cauchy identity as certain pairs of chains in Young's lattice.  The left hand side of \eqref{eq:Cauchy0} is a weighted sum over pairs of chains $((P^{(i)})_i,(Q^{(j)})_j)$ of the form\footnote{For simplicity we explain the case of symmetric polynomials, i.e., the restriction to $\x=(x_1,\ldots,x_n)$ and $\y=(y_1,\ldots,y_m)$.}
\begin{equation}
\label{eq:chain 1}
\emptyset = P^{(0)} \prec \cdots \prec P^{(n)} = \la = Q^{(m)} \succ \cdots \succ Q^{(0)} = \emptyset
\end{equation}
and the right hand side of \eqref{eq:Cauchy0} is the shown product times the weight of the trivial pair $((P^{(i)})_i,(Q^{(j)})_j)=((\emptyset)_i,(\emptyset)_j)$ which satisfies
\begin{equation}
\label{eq:chain 2}
\emptyset = Q^{(m)} \succ \cdots \succ Q^{(0)} = P^{(0)} \prec \cdots \prec P^{(n)} = \emptyset.
\end{equation}
For precise definitions see Section~\ref{sec:zwei}. 
The second observation is that we can inductively transform each pair of chains of the first form into the trivial pair of chains by using \emph{local growth rules} and thereby gain the wanted product of the right hand side of \eqref{eq:Cauchy0}.

Both observations can be reformulated algebraically by using up and down operators. Firstly we can express the Schur polynomial $s_\la(\x)$ (resp., $s_\la(\y)$) by using up operators $U_{x_i}$ (resp., down operators $D_{y_j}$).
This implies that the weighted sum over all pairs of chains of the form as in \eqref{eq:chain 1} is equal to
\begin{equation}
\label{eq:chain 3}
\left\langle D_{y_1} \cdots D_{y_m} U_{x_n} \cdots U_{x_1} \emptyset,\emptyset  \right\rangle,
\end{equation}
while the right hand side of \eqref{eq:Cauchy0} times the weight of the trivial chain as in \eqref{eq:chain 2} is equal to
\begin{equation}
\label{eq:chain 4}
\prod_{\substack{1 \leq i \leq n \\ 1 \leq j \leq m}}\frac{1}{1-x_iy_j} \left\langle  U_{x_n} \cdots U_{x_1} D_{y_1} \cdots D_{y_m} \emptyset,\emptyset  \right\rangle.
\end{equation}
The second observation can be reformulated as a commutation relation of the operators $U_{x_i}$ and $D_{y_j}$ stating that we obtain a factor $(1-x_iy_j)^{-1}$ when commuting these operators.
It is not difficult to see that \eqref{eq:chain 4} can be obtained from \eqref{eq:chain 3} by iteratively commuting  all up with all down operators. In order to obtain a bijective correspondence for the Cauchy identity, it therefore suffices to find a bijective proof of the commutation relation which is typically called local growth rules.
In the case of the (dual) Cauchy identity for Schur polynomials, this approach leads to an infinite family of bijective proofs.
More generally, Fomin's approach yields a framework for (bijective) proofs of Cauchy type identities with many possibilities for generalisations by choosing different posets or weights, i.e., define different up and down operators. This approach was used for example by Bufetov and Matveev \cite{BufetovMatveev18} to obtain a combinatorial proof of the Cauchy identity for Hall--Littlewood polynomials or in \cite{FriedenSchreierAigner24plus} by Frieden and the author to obtain a generalisation of dual RSK for Macdonald polynomials.
\medskip

The motivation for this article is twofold. Firstly we present a concise summary of Fomin's framework for both the Cauchy and dual Cauchy identity for Schur polynomials and how it relates to the previous work by Knuth and Burge. For the case of the Cauchy identity this was already done by van Leeuwen \cite{vanLeeuwen05}, however the approach there is to define the local growth rules for RSK recursively by using the local growth rules of the Robinson correspondence, i.e., the restriction of RSK to permutations, while we give a direct description.

The second goal of this paper is to extend the framework of growth diagrams such that it is applicable to Littlewood identities. As in the Cauchy case, the approach is a two step process. In the first step we obtain an algebraic proof by introducing new relations between the up and down operators which we call \emph{projection identities}. In a second step we provide bijective proofs for these identities which become the building blocks for the bijective proofs of the Littlewood identities.
\medskip

The paper is structured in the following way. In Section~\ref{sec:zwei} we present Fomin's framework for the (dual) Cauchy identity. In particular in Section~\ref{sec:up and downs} we introduce the up and down operators, their commutation relations and present bijective proofs thereof, i.e., the (dual) row and column insertions. In Section~\ref{sec:growths} (dual) growth diagrams are introduced. We show how they yield a bijective proof of the (dual) Cauchy identity and how the typical insertion descriptions of RSK can be obtained from local growth rules. In Section~\ref{sec:skew} we show how this framework can be extended easily for the skew (dual) Cauchy identity and how one can obtain the (dual) Pieri rule as an immediate consequence.
In Section~\ref{sec:drei} we extend the framework to Littlewood identities by presenting the projection identities and their combinatorial proofs in Section~\ref{sec:projection identities}. In Section~\ref{sec:triangular} we define triangular (dual) growth diagrams and show how they yield bijective proofs for the Littlewood identities.
In Section~\ref{sec:vier} we present the general framework for up and down operators and how Cauchy and Littlewood type identities follow from commutation relations and projection identities between these operators.

\section{The Cauchy identity and its dual}
\label{sec:zwei}
\subsection{Preliminaries}

A \deff{partition} $\la$ is a weakly decreasing sequence of positive integers $\la=(\la_1,\ldots,\la_k)$. We call $k$ the \deff{length} of $\la$ and the $\la_i$ its \deff{parts}. We say that $\la$ is a partition of $n=|\la|=\la_1+\cdots+\la_k$ which is denoted by $\la \vdash n$. The \deff{Young diagram} of $\la$ is a collection of left-justified boxes, also called \deff{cells}, where the $i$-th row from bottom\footnote{Note that we are using French convention; for English convention one replaces ``bottom''  by ``top''.} has $\la_i$ boxes. As usual we identity a partition with its Young diagram. The \deff{conjugate} $\la^\prime=(\la_1^\prime,\ldots,\la_{m}^\prime)$ of a partition $\la$ is obtained by reflecting the Young diagram of $\la$ along the diagonal $x=y$. For a partition $\la$ let $l=l(\la)$ be the length of the \deff{Durfee square} which is given by $l(\la)=\max_i(\la_i \geq i)$. The \deff{Frobenius notation} of a partition $\la$ is defined as the pair $(\la_1-1,\ldots,\la_l-l|\la_1^\prime-1,\ldots,\la_l^\prime-l)$.

\begin{ex}
The partition $\la=(6,5,3,3,1)$ with Frobenius notation $(5,3,0|4,2,1)$ and its conjugate $(5,4,4,2,2,1)$ with Frobenius notation $(4,2,1|5,3,0)$ are shown on the left and right respectively. In both cases the cell  $(4,2)$ is marked in blue; note that we use Cartesian coordinates for cells.
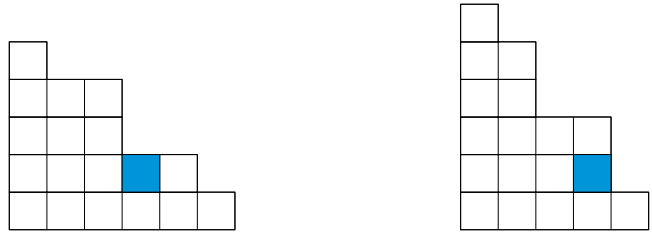
\begin{center}
\begin{tikzpicture}
\tyng(0cm,0cm,6,5,3,3,1)
\tyng(6cm,0cm,5,4,4,2,2,1)
\Ycyan
\tgyoung(0cm,0cm,,:::;)
\tgyoung(6cm,0cm,,:::;)
\end{tikzpicture}
\end{center}
\end{ex}

A filling of the cells of $\la$ with positive integers is called
\begin{itemize}
\item a \deff{semistandard Young tableau (SSYT)} of shape $\la$ if the entries are weakly increasing along rows and strictly increasing along columns,
\item a \deff{dual semistandard Young tableau} of shape $\la$ if the entries are strictly increasing along rows and weakly increasing along columns.
\end{itemize}

We denote by $\SSYT_\la(n)$ the set of semistandard Young tableaux of shape $\la$ with entries at most $n$ and by $\SSYT_\la^*(n)$ the set of dual SSYTs of shape $\la$ and entries at most $n$. It is immediate that SSYTs of shape $\la$ and dual SSYTs of shape $\la^\prime$ are in bijection by conjugating the Young diagram together with its entries. For a (dual) SSYT $T$ we denote by $\x^T$ the \deff{weight} of $T$ which is defined as 
\[
\x^T=\prod_{i \geq 1} x_i^{\#( i \text{ entries in }T)}.
\]

\begin{ex}
The following are a semistandard Young tableau (left) and a dual semistandard Young tableau (right), both of shape $(5,4,2,1)$ and with weight $x_1x_2^3x_3^2x_4^2x_5^4$.
\begin{center}
\young(12245,2335,45,5) \qquad \qquad
\young(12345,2345,25,5)
\end{center}
\end{ex}

For a sequence $\x=(x_1,\ldots, x_n)$ of variables, we define the \deff{Schur polynomial} $s_\la(\x)$ as
\[
s_\la(\x) = \sum_{T \in \SSYT_\la(n)}\x^T.
\]
The following is a fundamental theorem in the theory of symmetric polynomials.

\begin{thm}
\label{thm:Cauchy identities}
Let $\x=(x_1,\ldots,x_n)$ and $\y=(y_1,\ldots,y_m)$ be two sequences of variables. Then
\begin{align}
\label{eq:Cauchy}
\sum_{\la} s_\la(\x)s_\la(\y) &= \prod_{\substack{1 \leq i \leq n \\ 1 \leq j \leq m}} \frac{1}{1-x_iy_j},  &(\text{Cauchy identity})\\
\label{eq:dual Cauchy}
\sum_{\la} s_\la(\x) s_{\la^\prime}(\y) &= \prod_{\substack{1 \leq i \leq n \\ 1 \leq j \leq m}} (1+x_iy_j), &(\text{dual Cauchy identity})
\end{align}
where both sums are over all partitions $\la$.
\end{thm}

Both identities can be interpreted as identities between two generating functions as follows. The Cauchy identity is equal to
\[
\sum_{\la}\sum_{\substack{P \in \SSYT_\la(n) \\ Q \in \SSYT_\la(m)}} \x^P\y^Q = \sum_{A=(a_{i,j})} \prod_{i,j} (x_iy_j)^{a_{i,j}},
\]
where the left hand side is over all partitions $\la$ and the right hand side is over all $n \times m$ matrices $A$ of non-negative integers. 
The dual Cauchy identity on the other hand is equal to the identity
\[
\sum_{\la}\sum_{\substack{P \in \SSYT_\la(n) \\ Q^* \in \SSYT^*_\la(m)}} = \sum_{B=(b_{i,j})} \prod_{i,j} (x_iy_j)^{b_{i,j}},
\]
where the left hand side is over all partitions $\la$ and the right hand side is over all $n \times m$ $\{0,1\}$-matrices. We use this interpretation in Section~\ref{sec:growths} to provide a combinatorial proof.

\subsection{Up and (dual) down operators}
\label{sec:up and downs}
In this section we present an elementary algebraic framework which allows us to prove Theorem~\ref{thm:Cauchy identities}. We follow mostly the notations used in \cite{FriedenSchreierAigner24plus}, however there are certain changes since we aim to cover the setting for both RSK and dual RSK in this paper. In particular we add either a ``$*$'' or ``dual'' to notations used in \cite{FriedenSchreierAigner24plus} whenever these are specific for the dual RSK setting.
\\

For two partitions $\mu,\la$ we write $\mu \subseteq \la$ if the Young diagram of $\mu$ is contained in the Young diagram of $\la$. The poset defined by the inclusion relation $\subseteq$ is called \deff{Young's lattice}; its meet and join are $\cap$ and $\cup$, where $\la \cap \mu$ (resp., $\la \cup \mu$) is defined as the partition obtained by taking the intersection (resp., union) of the corresponding Young diagrams.
For $\mu \subseteq \la$ the \deff{skew diagram} $\la/\mu$ is defined as the Young diagram obtained by deleting all boxes of $\la$ which are part of $\mu$. We call the skew diagram $\la/\mu$ a \deff{horizontal strip} (resp., \deff{vertical strip}), denoted by $\mu \prec \la$ (resp., $\mu \prec^\prime \la$), if $\la/\mu$ contains at most one cell in each column (resp. row).
The partitions $\la,\mu$ are called \deff{interlacing} if $\mu \prec \la$.

We denote by $\bb{P}$ the set of partitions and define $\bb{Y}$ as the $\bb{Q}$-vector space generated by partitions, i.e.,
\[
\bb{Y}:=\bigoplus_{\la \in \bb{P}}\bb{Q}\la.
\]
Intuitively, this vector space already covers all combinatorial objects we are interested in for this paper. However to be algebraically precise, we need to actually work over the following more general space. For variables $x_1,\ldots,x_n$ we define the $\bb{Q}\llbracket x_1,\ldots,x_n\rrbracket$-module $\bb{Y}\llbracket x_1,\ldots,x_n \rrbracket$ as the set of formal power series in the variables $x_1,\ldots, x_n$ with coefficients in $\bb{Y}$ together with the usual addition and multiplication. For simplicity, we define for the rest of this paper all linear maps over the smallest reasonable field or algebra respectively and extend them linearly in the obvious way if necessary.\bigskip

The \deff{up operator} $U_x$, the \deff{down operator} $D_y$ and the \deff{dual down operator} $D_y^*$ are $\bb{Q}\llbracket x,y \rrbracket$-linear maps on $\bb{Y}\lb x,y \rb$ which are defined as 
\[
U_x \la = \sum_{\nu \succ \la} x^{|\nu/\la|} \nu, \qquad
D_y \la = \sum_{\mu \prec \la} y^{|\la/\mu|} \mu, \qquad
D_y^* \la=\sum_{\mu \prec^\prime \la} y^{|\la/\mu|} \mu,
\]
for a partition $\la$ and linearly extended onto $\bb{Y}\lb x,y \rb$.

A crucial property of these operators are the following commutation relations.
\begin{thm}
\label{thm:commutation relations}
The up and (dual) down operator satisfy the commutation relations
\begin{align}
\label{eq:commutation relation} D_y U_x = \frac{1}{1-xy} U_x D_y ,\\
\label{eq:dual commutation relation} D_y^* U_x = (1+x y) U_x D_y^*.
\end{align}
\end{thm}

Before proving this theorem, we use it to deduce the (dual) Cauchy identity.
First observe that each SSYT $T$ of shape $\la$ with entries at most $n$ can be interpreted as a chain of interlacing partitions
\[
\emptyset = T^{(0)} \prec T^{(1)} \prec \cdots \prec T^{(n)} = \la,
\]
where $T^{(i)}$ denotes the shape of the tableau $T$ when restricted to the entries at most $i$. Analogously a dual SSYT $R$ of shape $\la$ with entries at most $m$ can be interpreted as a chain of dual interlacing partitions
\[
\emptyset = R^{(0)} \prec^\prime R^{(1)} \prec^\prime \cdots \prec^\prime R^{(m)} = \la.
\]
We define the inner product $\langle \cdot , \cdot \rangle$ on the vector space $\bb{Y}$ as $\langle \la ,\rho \rangle = \delta_{\la,\rho}$ for all $\la,\rho \in \bb{P}$. For each $\rho \in \bb{P}$ we extend the inner product to a linear functional $\left\langle \cdot,\rho \right\rangle$ on $\bb{Y} \lb x_1,\ldots,x_n,y_1,\ldots,y_m\rb$ linearly. The above observations imply the following alternative formulas for the Schur polynomials $s_\la(\x), s_\la(\y)$ and $s_{\la^\prime}(\y)$
\begin{align}
\label{eq:schur as ups}
s_\la(\x) &= \left\langle U_{x_n} \cdots U_{x_1} \emptyset , \la \right\rangle, \\
\label{eq:schur as downs}
s_\la(\y) &= \left\langle D_{y_1} \cdots D_{y_m} \la , \emptyset \right\rangle, \\
\label{eq:schur as dual downs}
s_{\la^\prime}(\y) &= \left\langle D_{y_1}^* \cdots D_{y_m}^* \la , \emptyset \right\rangle,
\end{align}
where $\x=(x_1,\ldots,x_n)$ and $\y=(y_1,\ldots,y_m)$.
We calculate the following linear functional in two different ways; first as shown next
\begin{equation}
\label{eq:cauchy operator inner product}
\left\langle D_{y_m} \cdots D_{y_1} U_{x_n} \cdots U_{x_1} \emptyset , \emptyset \right\rangle
= \sum_\la \left\langle U_{x_n} \cdots U_{x_1} \emptyset , \la \right\rangle
  \left\langle D_{y_m} \cdots D_{y_1} \la , \emptyset \right\rangle
= \sum_\la s_\la(\x) s_\la(\y).
\end{equation}
By using the commutation relation \eqref{eq:commutation relation} repeatedly we obtain
\[
 D_{y_j} U_{x_n} \cdots U_{x_1} =
\frac{1}{1-x_n y_j} U_{x_n} D_{y_j} U_{x_{n-1}} \cdots U_{x_1}
= \prod_{i=1}^n \frac{1}{1-x_iy_j} U_{x_n} \cdots U_{x_1} D_{y_j}.
\]
Since $D_{y_j}\emptyset = \emptyset$ we obtain by using the above equation repeatedly
\[
\left\langle D_{y_m} \cdots D_{y_1} U_{x_n} \cdots U_{x_1} \emptyset,\emptyset \right\rangle = 
\prod_{\substack{1 \leq i \leq n \\ 1 \leq j \leq m}} \frac{1}{1-x_i y_j}
\left\langle U_{x_n} \cdots U_{x_1} \emptyset,\emptyset \right\rangle = 
\prod_{\substack{1 \leq i \leq n \\ 1 \leq j \leq m}} \frac{1}{1-x_i y_j},
\]
which proves together with \eqref{eq:cauchy operator inner product} the Cauchy identity \eqref{eq:Cauchy}. 
 In order to prove the dual Cauchy identity \eqref{eq:dual Cauchy} we regard the linear functional
$\left\langle D_{y_m}^* \cdots D_{y_1}^* U_{x_n} \cdots U_{x_1} \emptyset , \emptyset \right\rangle$
analogously.\bigskip

In the remainder of this section we provide a combinatorial proof of Theorem~\ref{thm:commutation relations}. For partitions $\la,\rho$ and a non-negative integer $k$ we define the sets
\begin{align*}
\mc{U}(\la,\rho,k) &:= \{ \nu: \la \prec \nu \succ \rho, |\nu /(\la \cup \rho)| = k\}, \\
\mc{D}(\la,\rho,k) &:= \{ \mu: \la \succ \mu \prec \rho, |(\la \cap \rho)/\mu| = k\}, \\
\mc{U}^*(\la,\rho,k) &:= \{ \nu: \la \prec^\prime \nu \succ \rho, |\nu /(\la \cup \rho)| = k\}, \\
\mc{D}^*(\la,\rho,k) &:= \{ \mu: \la \succ \mu \prec^\prime \rho, |(\la \cap \rho)/\mu| = k\}.
\end{align*}
Then \eqref{eq:commutation relation} is equivalent to the system of equations
\begin{equation}
\label{eq:commutator via sets}
\left| \mc{U}(\la,\rho,k) \right| = \sum_{i=0}^k \left| \mc{D}(\la,\rho,i) \right|,
\end{equation}
for all partitions $\la,\rho$ and non-negative integers $k$, and \eqref{eq:dual commutation relation} is equivalent to the system of equations 
\begin{equation}
\label{eq:dual commutator via sets}
\left| \mc{U}^*(\la,\rho,k) \right| = \left| \mc{D}^*(\la,\rho,k) \right| + \left| \mc{D}^*(\la,\rho,k-1) \right|,
\end{equation}
for all partitions $\la,\rho$ and non-negative integers $k$. 
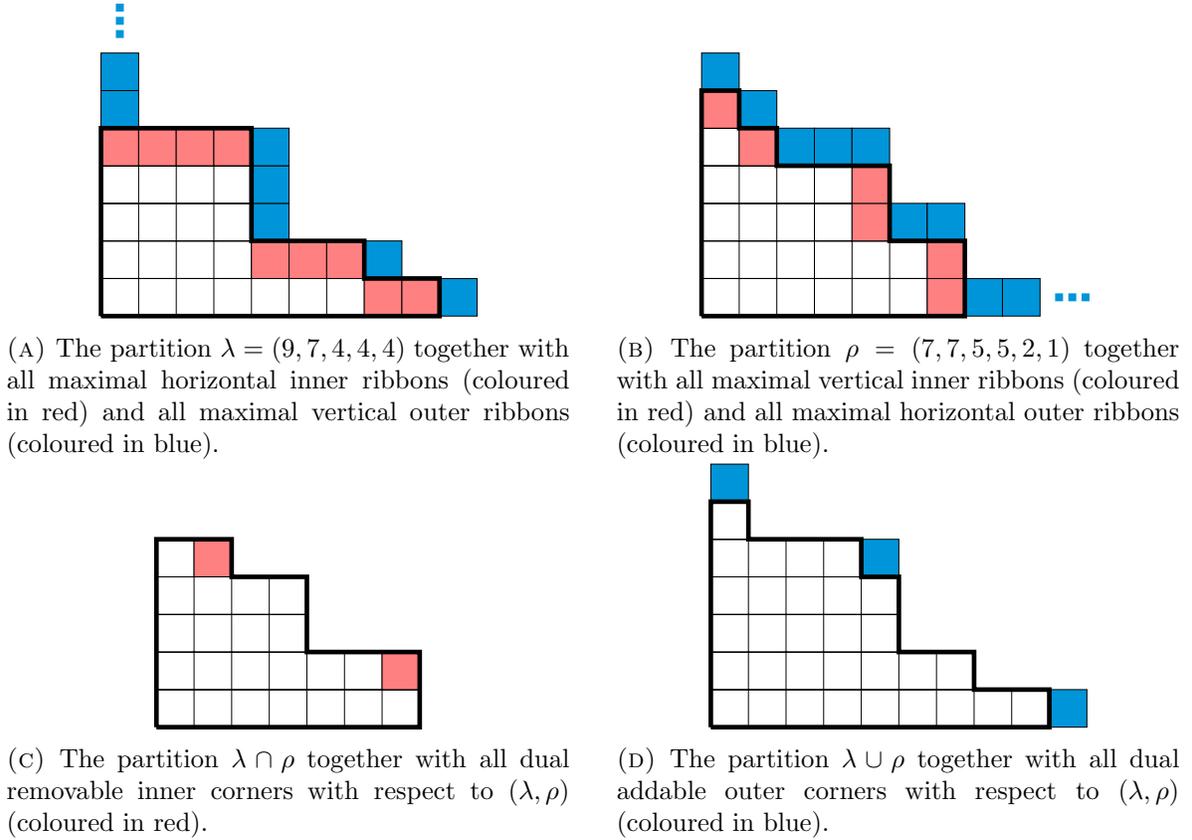
\begin{figure}
\begin{center}
\begin{subfigure}[t]{.48\textwidth}
\centering
\begin{tikzpicture}
\tyng(0,0,9,7,4,4,4)
\Yred
\tgyoung(0cm,0cm,:::::::;;,::::;;;,,,;;;;)
\Ycyan
\tgyoung(0cm,0cm,:::::::::;,:::::::;,::::;,::::;,::::;,;,;)
\draw[line width=1.75pt] (0,0) -- (4.5,0) -- (4.5,.5) -- (3.5,.5) -- (3.5,1) -- (2,1) -- (2,2.5) -- (0,2.5) -- (0,0);
\draw[line width=3pt, dotted, myblue] (.25,3.7) -- (.25,4.2);
\end{tikzpicture}
\caption{The partition $\la=(9,7,4,4,4)$ together with all maximal horizontal inner ribbons (coloured in red) and all maximal vertical outer ribbons (coloured in blue).}
\end{subfigure} \hfill
\begin{subfigure}[t]{.48\textwidth}
\centering
\begin{tikzpicture}
\tyng(0,0,7,7,5,5,2,1)
\Yred
\tgyoung(0cm,0cm,::::::;,::::::;,::::;,::::;,:;,;)
\Ycyan
\tgyoung(0cm,0cm,:::::::;;,,:::::;;,,::;;;,:;,;)
\draw[line width=1.75pt] (0,0) -- (3.5,0) -- (3.5,1) -- (2.5,1) -- (2.5,2) -- (1,2) -- (1,2.5) -- (.5,2.5) -- (.5,3) -- (0,3) -- (0,0);
\draw[line width=3pt, dotted, myblue] (4.7,.25) -- (5.2,.25);
\end{tikzpicture}
\caption{The partition $\rho=(7,7,5,5,2,1)$ together with all maximal vertical inner ribbons (coloured in red) and all maximal horizontal outer ribbons (coloured in blue).}
\end{subfigure}
\begin{subfigure}[t]{.48\textwidth}
\centering
\begin{tikzpicture}
\tyng(0,0,7,7,4,4,2)
\Yred
\tgyoung(0cm,0cm,,::::::;,,,:;)
\draw[line width=1.75pt] (0,0) -- (3.5,0) -- (3.5,1) -- (2,1) -- (2,2) -- (1,2) -- (1,2.5) -- (0,2.5) -- (0,0);
\end{tikzpicture}
\caption{The partition $\la \cap \rho$ together with all dual removable inner corners with respect to $(\la,\rho)$ (coloured in red).}
\end{subfigure} \hfill
\begin{subfigure}[t]{.48\textwidth}
\centering
\begin{tikzpicture}
\tyng(0,0,9,7,5,5,4,1)
\Ycyan
\tgyoung(0cm,0cm,:::::::::;,,,,::::;,,;)
\draw[line width=1.75pt] (0,0) -- (4.5,0) -- (4.5,.5) -- (3.5,.5) -- (3.5,1) -- (2.5,1) -- (2.5,2) -- (2,2) -- (2,2.5) -- (.5,2.5) -- (.5,3) -- (0,3) -- (0,0);
\end{tikzpicture}
\caption{The partition $\la \cup \rho$ together with all dual addable outer corners  with respect to $(\la,\rho)$ (coloured in blue).}
\end{subfigure}

\caption{\label{fig:dual removable and addable} The stepwise construction of the removable inner corners of $\la\cap \rho$ and addable outer corners of $\la \cup \rho$.
}
\end{center}
\end{figure}
In order to prove both systems of equations combinatorially, we need to introduce further notations.
\begin{itemize}
\item A \deff{horizontally inner ribbon} of $\la$ is set of cells $\{(i,j),(i+1,j),\ldots,(i+k,j)\}=\la/\mu$ such that $\mu \prec \la$.
\item A \deff{vertically inner ribbon} of $\la$ is set of cells $\{(i,j),(i,j+1),\ldots,(i,j+k)\}=\la/\mu$ such that $\mu \prec^\prime \la$.
\item A \deff{horizontally outer ribbon} of $\la$ is set of cells $\{(i,j),(i+1,j),\ldots,(i+k,j)\}=\nu/\la$ such that $\la \prec \nu$.
\item A \deff{vertically outer ribbon} of $\la$ is set of cells $\{(i,j),(i,j+1),\ldots,(i,j+k)\}=\nu/\la$ such that $\la \prec^\prime \nu$.
\end{itemize}
Note that the above definitions generalise the notions of inner and outer corners. An \deff{inner corner} is an inner ribbon that is both horizontally and vertically, and analogously an \deff{outer corner} is an outer ribbon that is both horizontally and vertically.
We call a horizontal inner ribbon of $\la \cap \rho$ \deff{removable} with respect to $(\la,\rho)$ if it is the intersection of a horizontal inner ribbon of $\la$ and a horizontal inner ribbon of $\rho$. Analogously a horizontal outer ribbon of $\la \cup \rho$ is called \deff{addable} with respect to $(\la,\rho)$ if it is the intersection of a horizontal outer ribbon of $\la$ and a horizontal outer ribbon of $\rho$. For the dual RSK setting, we call an inner corner of $\la \cap \rho$ \deff{dual removable} with respect to $(\la,\rho)$ if it is the intersection of a horizontal inner ribbon of $\la$ and a vertical inner ribbon of $\rho$, and analogously we call an outer corner of $\la \cup \rho$ \deff{dual addable} with respect to $(\la,\rho)$ if it is the intersection of vertical outer ribbon of $\la$ and a horizontal outer ribbon of $\rho$. See Figure~\ref{fig:dual removable and addable} for an example of dual addable and dual removable corners and Figure~\ref{fig:row and column insertion} for an example of addable and removable ribbons. In all of the above settings we omit referring to $(\la,\rho)$ whenever they are clear from context.

Let $\la, \rho$ be two partitions and denote by $d$ the number of removable ribbons of $\la \cap \rho$ of size $1$, i.e. $d$ is the size of $\mc{D}(\la,\rho,1)$. We say that a removable ribbon of $\la \cap \rho$ is \deff{in position $i$} if it is in the same row as the $i$-th element of $\mc{D}(\la,\rho,1)$ when counted from bottom starting with $1$. Analogously we say that an addable ribbon of $\la\cup \rho$ is in position $i$ if in the same row as the $i$-th element of $\mc{U}(\la,\rho,1)$ counted from bottom and starting with $0$.

We identify elements of $\mu \in \mc{D}(\la,\rho,k)$ and $\nu \in \mc{U}(\la,\rho,k)$ with multisets $\Ret(\mu)$ of $[d]=\{1,2,\ldots,d\}$ or $\Set(\nu)$ of $[0,d]=\{0,1,\ldots,d\}$ respectively as follows.
Given an element $\mu \in \mc{D}(\la,\rho,k)$ we can decompose $(\la \cap \rho)/\mu$ into a set $\{R_{i_1},\ldots,R_{i_l}\}$ of removable ribbons of $\la \cap \rho$ such that the $R_{i_j}$ are pairwise different, $R_{i_j}$ is in position $i_j$, and that 
\[
(\la \cap \rho) / \mu = \bigcup_{i=1}^l R_i.
\]
We define the multiset  $\Ret(\mu)$ as the multiset that contains the element $i_j$ with multiplicity $|R_{i_j}|$ for $1 \leq j \leq l$.
The construction is analogue for $\Set(\nu)$. We can apply the same constructions for the elements of $\mu^* \in \mc{D}^*(\la,\rho,k)$ and $\nu^* \in \mc{U}^*(\la,\rho,k)$ with the only difference that $\Ret(\mu^*)$ is a $k$-subset of $[d]$ and that $\Set(\nu^*)$ is a $k$-subset of $[0,d]$.
In the following we construct for each $\la,\rho, k$ the natural bijections $\F_{\la,\rho,k}^\row, \F_{\la,\rho,k}^\col, \F_{\la,\rho,k}^{\row*}, \F_{\la,\rho,k}^{\col*}$,
\begin{align*}
\F_{\la,\rho,k}^{\bullet}: \bigcup_{i=0}^k \mc{D}(\la,\rho,i) \rightarrow \mc{U}(\la,\rho,k), \qquad
\F_{\la,\rho,k}^{\bullet*}:  \mc{D}^*(\la,\rho,k) \cup \mc{D}^*(\la,\rho,k-1) \rightarrow \mc{U}^*(\la,\rho,k),
\end{align*}
where $\bullet$ can be either $\row$ or $\col$ and hence prove Theorem~\ref{thm:commutation relations} bijectively.\\

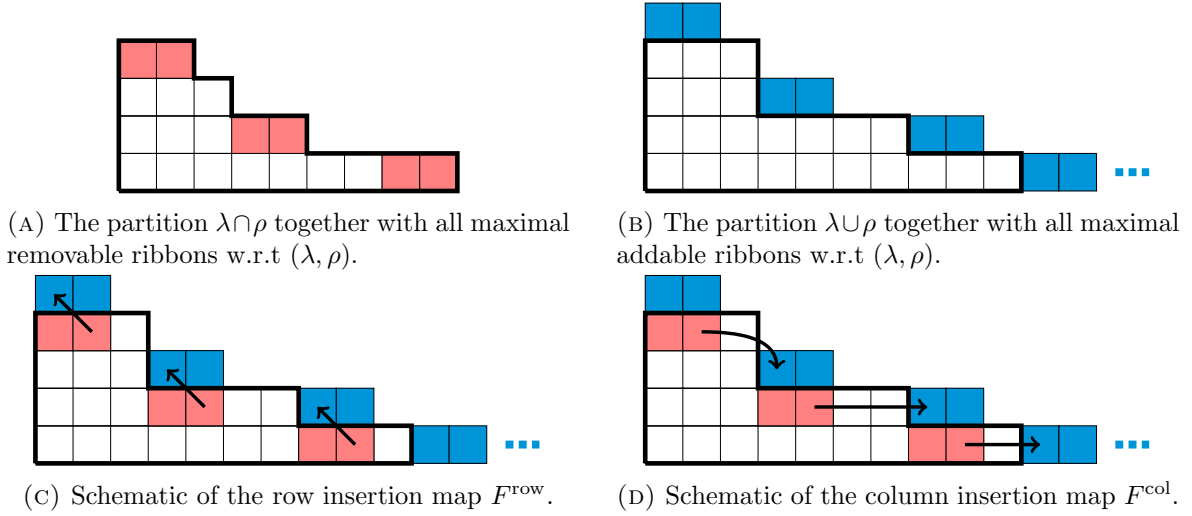
\begin{figure}
\begin{center}
\begin{subfigure}[t]{.48\textwidth}
\centering
\begin{tikzpicture}
\tyng(0,0,9,5,3,2)
\Yred
\tgyoung(0cm,0cm,:::::::;;,:::;;,,;;)
\draw[line width=1.75pt] (0,0) -- (4.5,0) -- (4.5,.5) -- (2.5,.5) -- (2.5,1) -- (1.5,1) -- (1.5,1.5) -- (1,1.5) -- (1,1.5) -- (1,2) -- (0,2) -- (0,0);
\end{tikzpicture}
\caption{The partition $\la \cap \rho$ together with all maximal removable ribbons w.r.t $(\la,\rho)$.}
\end{subfigure} \hfill
\begin{subfigure}[t]{.48\textwidth}
\centering
\begin{tikzpicture}
\tyng(0,0,10,7,3,3)
\Ycyan
\tgyoung(0cm,0cm,::::::::::;;,:::::::;;,:::;;,,;;)
\draw[line width=1.75pt] (0,0) -- (5,0) -- (5,.5) -- (3.5,.5) -- (3.5,1) -- (1.5,1) -- (1.5,2) -- (0,2) -- (0,0);
\draw[line width=3pt, dotted, myblue] (6.25,.25) -- (6.7,.25);
\end{tikzpicture}
\caption{The partition $\la \cup \rho$ together with all maximal addable ribbons w.r.t $(\la,\rho)$.}
\end{subfigure}
\begin{subfigure}[t]{.48\textwidth}
\centering
\begin{tikzpicture}
\tyng(0,0,10,7,3,3)
\Ycyan
\tgyoung(0cm,0cm,::::::::::;;,:::::::;;,:::;;,,;;)
\Yred
\tgyoung(0cm,0cm,:::::::;;,:::;;,,;;)
\draw[line width=3pt, dotted, myblue] (6.25,.25) -- (6.7,.25);
\draw[line width=1.75pt] (0,0) -- (5,0) -- (5,.5) -- (3.5,.5) -- (3.5,1) -- (1.5,1) -- (1.5,2) -- (0,2) -- (0,0);

\draw [->,line width=1.3pt] (4.25,.25) -- (3.75,.75);
\draw [->,line width=1.3pt] (2.25,.75) -- (1.75,1.25);
\draw [->,line width=1.3pt] (.75,1.75) -- (.25,2.25);

\end{tikzpicture}
\caption{Schematic of the row insertion map $F^\row$.}
\end{subfigure} \hfill
\begin{subfigure}[t]{.48\textwidth}
\centering
\begin{tikzpicture}
\tyng(0,0,10,7,3,3)
\Ycyan
\tgyoung(0cm,0cm,::::::::::;;,:::::::;;,:::;;,,;;)
\Yred
\tgyoung(0cm,0cm,:::::::;;,:::;;,,;;)
\draw[line width=3pt, dotted, myblue] (6.25,.25) -- (6.7,.25);
\draw[line width=1.75pt] (0,0) -- (5,0) -- (5,.5) -- (3.5,.5) -- (3.5,1) -- (1.5,1) -- (1.5,2) -- (0,2) -- (0,0);

\draw [->,line width=1.3pt] (4.25,.25) -- (5.25,.25);
\draw [->,line width=1.3pt] (2.25,.75) -- (3.75,.75);
\draw [->,line width=1.3pt] (.75,1.75) to [out=0,in=90] (1.75,1.25);
\end{tikzpicture}
\caption{Schematic of the column insertion map $F^\col$.}
\end{subfigure}
\caption{\label{fig:row and column insertion} The inner and outer horizontal ribbons of $\la \cap \rho$ or $\la \cup \rho$ respectively and the schematics of the row and column insertions for $\la=(10,5,3,2)$ and $\rho=(9,7,3,3)$.}
\end{center}
\end{figure}

The bijection $\F_{\la,\rho,k}^\row$, called \deff{row insertion} maps a multiset $\Ret \in \mc{D}(\la,\rho,j)$ to the set $\Ret \cup \{ 0^{(k-j)}\}$. The definition of the \deff{column insertion}\footnote{Note that column insertion is sometimes also called \deff{Burge insertion}, compare for example to {\cite[§3.2]{vanLeeuwen05}}.} $\F_{\la,\rho,k}^\col$ is slightly more complicated. Denote by $\Set^\infty$ the union of all addable ribbons of $\la \cup \rho$. In order to obtain $\F_{\la,\rho,k}^\col(\Ret)$ we construct a sequence $(A_i,B_i)$ of pairs of multisets as follows. For $\Ret \in \mc{D}(\la,\rho,j)$ start with the pair of multisets $(A_0,B_0)=(\emptyset,\Set^\infty \setminus \Ret)$ and denote by $x_1 \leq \cdots \leq x_k$ the elements of $\Ret \cup \{\infty^{(k-j)}\}$. In the $i$-th step set $(A_i,B_i)=(A_{i-1}\cup\{y\},B_{i-1}\setminus\{y\})$ where $y$ is the largest element in $B_{i-1}$ which is smaller than $x_i$. We define $\F_{\la,\rho,k}^\col(\Ret)$ as the terminal multiset $A_k$. For an alternative description of column insertion see for example \cite[§3.2]{vanLeeuwen05}. Both maps are illustrated in Figure~\ref{fig:row and column insertion}.

\begin{ex}
The following table shows the row and column insertion maps for $\la=\rho=(3,2)$ and $k=2$. The top row contains the elements $\mu \in \bigcup_{i=0}^2 \mc{D}(\la,\la,i)$, where $\mu$ is obtained by removing the red boxes. The two bottom rows contain the corresponding images $\nu \in \mc{U}(\la,\la,2)$ where each $\nu$ is obtained by adding the blue boxes to $\la$.
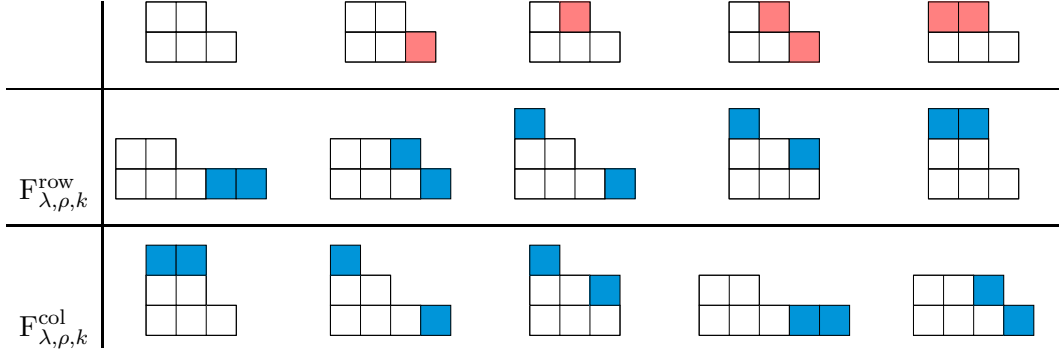
\begin{center}
\Yboxdim{0.4cm}
\begin{tabular}{c|ccccccccc}
& \YoungRed{3,2}{} && \YoungRed{3,2}{::;} && \YoungRed{3,2}{,:;} && \YoungRed{3,2}{::;,:;} && \YoungRed{3,2}{,;;} \\[6pt] \hline \\[-6pt]
$\F_{\la,\rho,k}^\row$ & \YoungBlue{3,2}{:::;;} && \YoungBlue{3,2}{:::;,::;} && \YoungBlue{3,2}{:::;,,;} && \YoungBlue{3,2}{,::;,;} && \YoungBlue{3,2}{,,;;}  \\[6pt] \hline \\[-6pt]
$\F_{\la,\rho,k}^\col$ &  \YoungBlue{3,2}{,,;;} && \YoungBlue{3,2}{:::;,,;} && \YoungBlue{3,2}{,::;,;} && \YoungBlue{3,2}{:::;;} && \YoungBlue{3,2}{:::;,::;} 
\end{tabular}
\Yboxdim{0.5cm}
\end{center}
\end{ex}

For the dual setting the definitions become simpler since we have sets instead of multisets. The bijections $\F_{\la,\rho,k}^{\row*}$, called \deff{dual row insertion}, and $\F_{\la,\rho,k}^{\col*}$, called \deff{dual column insertion}, are defined as
\[
\F_{\la,\rho,k}^{\row*}(\Ret) = \begin{cases}
\Ret \qquad & |\Ret|=k,\\
\Ret \cup\{0\} & |\Ret|=k-1,
\end{cases} \qquad
\F_{\la,\rho,k}^{\col*}(\Ret) = \begin{cases}
\{x-1:x\in \Ret\} \quad & |\Ret|=k,\\
\{x-1:x\in \Ret\} \cup\{d\} & |\Ret|=k-1,
\end{cases}
\]

\subsection{Growth diagrams}
\label{sec:growths}
In this subsection we define Fomin's growth diagrams \cite{Fomin86, Fomin95} which allow us to transform a bijective proof of the commutation relation into a bijective proof of the (dual) Cauchy identity. 
In the following we view an $n\times m$ matrix as an $n\times m$ grid of squares for which we label the vertices of the grid by partitions. The coordinates of the squares lie in $[n] \times [m]$, and the coordinates of the vertices in $[0,n] \times [0,m]$. In both cases these coordinates are interpreted as matrix coordinates instead of Cartesian coordinates.

\begin{defi}
\label{def:growth}
\begin{enumerate}
\item Let $A$ be an $n \times m$ matrix with non-negative entries. A \deff{growth associated with $A$} is an assignment $\Lambda=(\Lambda_{i,j})$ of partitions to the vertices $(i,j)\in [0,n]\times[0,m]$ such that
\begin{itemize}
\item for $i<n$ holds $\Lambda_{i,j} \prec \Lambda_{i+1,j}$, i.e., $\Lambda_{i+1,j}/\Lambda_{i,j}$ is a horizontal strip,
\item for $j<m$ holds $\Lambda_{i,j} \prec \Lambda_{i,j+1}$, i.e., $\Lambda_{i,j+1}/\Lambda_{i,j}$ is a horizontal strip,
\item $|\Lambda_{i,j}|$ is equal to the sum of entries in $A$ to the north-west of $(i,j)$, that is,
\begin{equation}
\label{eq:growth sizes}
|\Lambda_{i,j}| = \sum_{\substack{1 \leq k \leq i \\ 1 \leq l \leq j}}A_{k,l}.
\end{equation}
\end{itemize}

\item Let $B$ be an $n \times m$ $\{0,1\}$-matrix. A \deff{dual growth associated with $B$} is an assignment $\Lambda=(\Lambda_{i,j})$ of partitions to the vertices $(i,j)\in [0,n]\times[0,m]$ such that
\begin{itemize}
\item for $i<n$ holds $\Lambda_{i,j} \prec \Lambda_{i+1,j}$, i.e., $\Lambda_{i+1,j}/\Lambda_{i,j}$ is a horizontal strip,
\item for $j<m$ holds $\Lambda_{i,j} \prec^\prime \Lambda_{i,j+1}$, i.e., $\Lambda_{i,j+1}/\Lambda_{i,j}$ is a vertical strip,
\item $|\Lambda_{i,j}|$ is equal to the sum of entries in $B$ to the north-west of $(i,j)$, that is,
\[
|\Lambda_{i,j}| = \sum_{\substack{1 \leq k \leq i \\ 1 \leq l \leq j}}B_{k,l}.
\]
\end{itemize}
\end{enumerate}
\end{defi}
Since the matrices $A$ or $B$ respectively are determined by $\Lambda$ we will often omit the reference to it and call $\Lambda$ a \deff{growth} or \deff{dual growth} respectively.

\begin{ex}
\label{ex:growths}
The matrix $A= \begin{pmatrix}
0 & 1 \\ 1 & 0 \\ 1 & 1
\end{pmatrix}$ has four growths associated to it as shown next,

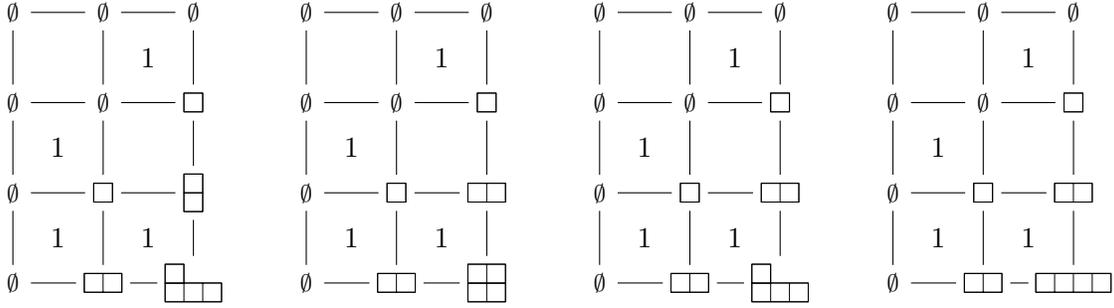
\begin{center}
\begin{tikzpicture}[scale=1.2]

	\node at (1.5,-0.5) {$1$};
	\node at (0.5,-1.5) {$1$};
	\node at (0.5,-2.5) {$1$};
	\node at (1.5,-2.5) {$1$};
	{\footnotesize
		\node at (0,0) {$\emptyset$};
		\node at (1,0) {$\emptyset$};
		\node at (2,0) {$\emptyset$};
		\node at (0,-1) {$\emptyset$};
		\node at (0,-2) {$\emptyset$};
		\node at (0,-3) {$\emptyset$};
		\node at (1,-1) {$\emptyset$};
	\Yboxdim{0.25cm}
		\node at (1,-2) {\yng(1)};
		\node at (1,-3) {\yng(2)};
		\node at (2,-1) {\yng(1)};
		\node at (2,-2) {\yng(1,1)};
		\node at (2,-3) {\yng(3,1)};
	}
	\draw (.2,0) -- (.8,0);
	\draw (.2,-1) -- (.8,-1);
	\draw (.2,-2) -- (.8,-2);
	\draw (.2,-3) -- (.7,-3);
	\draw (1.2,0) -- (1.8,0);
	\draw (1.2,-1) -- (1.8,-1);
	\draw (1.2,-2) -- (1.8,-2);
	\draw (1.3,-3) -- (1.6,-3);
	
	\draw (0,-.2) -- (0,-.8);
	\draw (0,-1.2) -- (0,-1.8);
	\draw (0,-2.2) -- (0,-2.8);
	\draw (1,-.2) -- (1,-.8);
	\draw (1,-1.2) -- (1,-1.8);
	\draw (1,-2.2) -- (1,-2.8);
	\draw (2,-.2) -- (2,-.8);
	\draw (2,-1.2) -- (2,-1.7);
	\draw (2,-2.3) -- (2,-2.7);

	\begin{scope}[xshift=3.25cm]
	\node at (1.5,-0.5) {$1$};
	\node at (0.5,-1.5) {$1$};
	\node at (0.5,-2.5) {$1$};
	\node at (1.5,-2.5) {$1$};
	{\footnotesize
		\node at (0,0) {$\emptyset$};
		\node at (1,0) {$\emptyset$};
		\node at (2,0) {$\emptyset$};
		\node at (0,-1) {$\emptyset$};
		\node at (0,-2) {$\emptyset$};
		\node at (0,-3) {$\emptyset$};
		\node at (1,-1) {$\emptyset$};
	\Yboxdim{0.25cm}
		\node at (1,-2) {\yng(1)};
		\node at (1,-3) {\yng(2)};
		\node at (2,-1) {\yng(1)};
		\node at (2,-2) {\yng(2)};
		\node at (2,-3) {\yng(2,2)};
	}
	\draw (.2,0) -- (.8,0);
	\draw (.2,-1) -- (.8,-1);
	\draw (.2,-2) -- (.8,-2);
	\draw (.2,-3) -- (.7,-3);
	\draw (1.2,0) -- (1.8,0);
	\draw (1.2,-1) -- (1.8,-1);
	\draw (1.2,-2) -- (1.7,-2);
	\draw (1.3,-3) -- (1.7,-3);
	
	\draw (0,-.2) -- (0,-.8);
	\draw (0,-1.2) -- (0,-1.8);
	\draw (0,-2.2) -- (0,-2.8);
	\draw (1,-.2) -- (1,-.8);
	\draw (1,-1.2) -- (1,-1.8);
	\draw (1,-2.2) -- (1,-2.8);
	\draw (2,-.2) -- (2,-.8);
	\draw (2,-1.2) -- (2,-1.8);
	\draw (2,-2.2) -- (2,-2.7);
	\end{scope}

	\begin{scope}[xshift=6.5cm]
	\node at (1.5,-0.5) {$1$};
	\node at (0.5,-1.5) {$1$};
	\node at (0.5,-2.5) {$1$};
	\node at (1.5,-2.5) {$1$};
	{\footnotesize
		\node at (0,0) {$\emptyset$};
		\node at (1,0) {$\emptyset$};
		\node at (2,0) {$\emptyset$};
		\node at (0,-1) {$\emptyset$};
		\node at (0,-2) {$\emptyset$};
		\node at (0,-3) {$\emptyset$};
		\node at (1,-1) {$\emptyset$};
	\Yboxdim{0.25cm}
		\node at (1,-2) {\yng(1)};
		\node at (1,-3) {\yng(2)};
		\node at (2,-1) {\yng(1)};
		\node at (2,-2) {\yng(2)};
		\node at (2,-3) {\yng(3,1)};
	}
	\draw (.2,0) -- (.8,0);
	\draw (.2,-1) -- (.8,-1);
	\draw (.2,-2) -- (.8,-2);
	\draw (.2,-3) -- (.7,-3);
	\draw (1.2,0) -- (1.8,0);
	\draw (1.2,-1) -- (1.8,-1);
	\draw (1.2,-2) -- (1.7,-2);
	\draw (1.3,-3) -- (1.6,-3);
	
	\draw (0,-.2) -- (0,-.8);
	\draw (0,-1.2) -- (0,-1.8);
	\draw (0,-2.2) -- (0,-2.8);
	\draw (1,-.2) -- (1,-.8);
	\draw (1,-1.2) -- (1,-1.8);
	\draw (1,-2.2) -- (1,-2.8);
	\draw (2,-.2) -- (2,-.8);
	\draw (2,-1.2) -- (2,-1.8);
	\draw (2,-2.2) -- (2,-2.7);
	\end{scope}

	\begin{scope}[xshift=9.75cm]
	\node at (1.5,-0.5) {$1$};
	\node at (0.5,-1.5) {$1$};
	\node at (0.5,-2.5) {$1$};
	\node at (1.5,-2.5) {$1$};
	{\footnotesize
		\node at (0,0) {$\emptyset$};
		\node at (1,0) {$\emptyset$};
		\node at (2,0) {$\emptyset$};
		\node at (0,-1) {$\emptyset$};
		\node at (0,-2) {$\emptyset$};
		\node at (0,-3) {$\emptyset$};
		\node at (1,-1) {$\emptyset$};
	\Yboxdim{0.25cm}
		\node at (1,-2) {\yng(1)};
		\node at (1,-3) {\yng(2)};
		\node at (2,-1) {\yng(1)};
		\node at (2,-2) {\yng(2)};
		\node at (2,-3) {\yng(4)};
	}
	\draw (.2,0) -- (.8,0);
	\draw (.2,-1) -- (.8,-1);
	\draw (.2,-2) -- (.8,-2);
	\draw (.2,-3) -- (.7,-3);
	\draw (1.2,0) -- (1.8,0);
	\draw (1.2,-1) -- (1.8,-1);
	\draw (1.2,-2) -- (1.7,-2);
	\draw (1.3,-3) -- (1.5,-3);
	
	\draw (0,-.2) -- (0,-.8);
	\draw (0,-1.2) -- (0,-1.8);
	\draw (0,-2.2) -- (0,-2.8);
	\draw (1,-.2) -- (1,-.8);
	\draw (1,-1.2) -- (1,-1.8);
	\draw (1,-2.2) -- (1,-2.8);
	\draw (2,-.2) -- (2,-.8);
	\draw (2,-1.2) -- (2,-1.8);
	\draw (2,-2.2) -- (2,-2.8);
	\end{scope}
\end{tikzpicture}
\end{center}
and two dual growths associated to it which are depicted below. For readability we omit the $0$ entries in the (dual) growths.
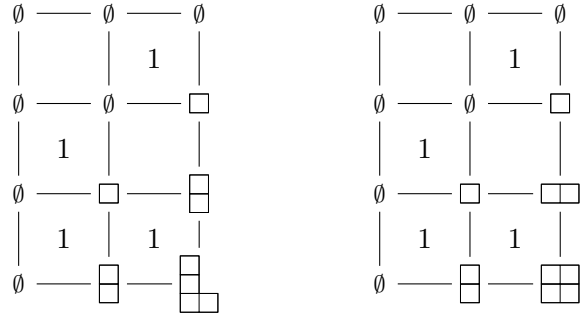
\begin{center}
\begin{tikzpicture}[scale=1.2]

	\node at (1.5,-0.5) {$1$};
	\node at (0.5,-1.5) {$1$};
	\node at (0.5,-2.5) {$1$};
	\node at (1.5,-2.5) {$1$};
	{\footnotesize
		\node at (0,0) {$\emptyset$};
		\node at (1,0) {$\emptyset$};
		\node at (2,0) {$\emptyset$};
		\node at (0,-1) {$\emptyset$};
		\node at (0,-2) {$\emptyset$};
		\node at (0,-3) {$\emptyset$};
		\node at (1,-1) {$\emptyset$};
	\Yboxdim{0.25cm}
		\node at (1,-2) {\yng(1)};
		\node at (1,-3) {\yng(1,1)};
		\node at (2,-1) {\yng(1)};
		\node at (2,-2) {\yng(1,1)};
		\node at (2,-3) {\yng(2,1,1)};
	}
	\draw (.2,0) -- (.8,0);
	\draw (.2,-1) -- (.8,-1);
	\draw (.2,-2) -- (.8,-2);
	\draw (.2,-3) -- (.8,-3);
	\draw (1.2,0) -- (1.8,0);
	\draw (1.2,-1) -- (1.8,-1);
	\draw (1.2,-2) -- (1.8,-2);
	\draw (1.2,-3) -- (1.7,-3);
	
	\draw (0,-.2) -- (0,-.8);
	\draw (0,-1.2) -- (0,-1.8);
	\draw (0,-2.2) -- (0,-2.8);
	\draw (1,-.2) -- (1,-.8);
	\draw (1,-1.2) -- (1,-1.8);
	\draw (1,-2.2) -- (1,-2.7);
	\draw (2,-.2) -- (2,-.8);
	\draw (2,-1.2) -- (2,-1.7);
	\draw (2,-2.3) -- (2,-2.6);

	\begin{scope}[xshift=4cm]
	\node at (1.5,-0.5) {$1$};
	\node at (0.5,-1.5) {$1$};
	\node at (0.5,-2.5) {$1$};
	\node at (1.5,-2.5) {$1$};
	{\footnotesize
		\node at (0,0) {$\emptyset$};
		\node at (1,0) {$\emptyset$};
		\node at (2,0) {$\emptyset$};
		\node at (0,-1) {$\emptyset$};
		\node at (0,-2) {$\emptyset$};
		\node at (0,-3) {$\emptyset$};
		\node at (1,-1) {$\emptyset$};
	\Yboxdim{0.25cm}
		\node at (1,-2) {\yng(1)};
		\node at (1,-3) {\yng(1,1)};
		\node at (2,-1) {\yng(1)};
		\node at (2,-2) {\yng(2)};
		\node at (2,-3) {\yng(2,2)};
	}
	\draw (.2,0) -- (.8,0);
	\draw (.2,-1) -- (.8,-1);
	\draw (.2,-2) -- (.8,-2);
	\draw (.2,-3) -- (.8,-3);
	\draw (1.2,0) -- (1.8,0);
	\draw (1.2,-1) -- (1.8,-1);
	\draw (1.2,-2) -- (1.7,-2);
	\draw (1.2,-3) -- (1.7,-3);
	
	\draw (0,-.2) -- (0,-.8);
	\draw (0,-1.2) -- (0,-1.8);
	\draw (0,-2.2) -- (0,-2.8);
	\draw (1,-.2) -- (1,-.8);
	\draw (1,-1.2) -- (1,-1.8);
	\draw (1,-2.2) -- (1,-2.7);
	\draw (2,-.2) -- (2,-.8);
	\draw (2,-1.2) -- (2,-1.8);
	\draw (2,-2.2) -- (2,-2.7);
	\end{scope}
\end{tikzpicture}
\end{center}
\end{ex}

Given a (dual) growth $\Lambda$ we define $P(\Lambda)$ as the semistandard Young tableau determined by the last column of $\Lambda$, i.e., the shape of $P(\Lambda)$ restricted to entries at most $i$ is $\Lambda_{i,m}$, and $Q(\Lambda)$ as the (dual) semistandard Young tableau determined by the last row of $\Lambda$, i.e., the shape of $Q(\Lambda)$ restricted to entries at most $j$ is $\Lambda_{n,j}$. The according tableaux for the third growth in Example~\ref{ex:growths} are
\[
P(\Lambda)=\young(123,3)\,, \qquad\qquad Q(\Lambda)=\young(112,2)\,.
\]
If $\Lambda$ is a (dual) growth associated with $A$ and $P=P(\Lambda),Q=Q(\Lambda)$, we write
\[
\Lambda: A \rightarrow (P,Q),
\]
and say that ``$\Lambda$ is a (dual) growth from $A$ to $(P,Q)$.  While in general, there may be multiple (dual) growths from $A$ to $(P,Q)$ the setting introduced next allows us to choose a unique (dual) growth.
Observe that each square of a growth has the form 
\begin{equation}
\label{eq:growth square}
\begin{tikzpicture}[baseline=-3.5ex]
\draw (-0.2,-.1) --node[left]{\rotatebox{-90}{$\prec$}} (-0.2,-.9);
\draw (1.2,-.1) --node[right]{\rotatebox{-90}{$\prec$}} (1.2,-.9);
\draw (0.1,0.2) --node[above]{$\prec$} (.9,0.2);
\draw (.1,-1.2) --node[below]{$\prec$} (.9,-1.2);
\node at (-0.2,0.2) {$\mu$};
\node at (1.2,0.2) {$\rho$};
\node at (-0.2,-1.15) {$\la$};
\node at (1.2,-1.2) {$\nu$};
\node at (.5,-.5) {$a$};
\end{tikzpicture}
\quad \text{with} \quad a \in \bb{N}, \quad \nu \in \mc{U}(\la,\rho,k), \quad \mu \in \mc{D}(\la,\rho,k-a),
\end{equation}
where $k=|\nu/(\la\cup\rho)|=|(\la\cap \rho)/\mu|+a$, and each square of a dual growth has the form
\begin{equation}
\label{eq:dual growth square}
\begin{tikzpicture}[baseline=-3.5ex]
\draw (-0.2,-.1) --node[left]{\rotatebox{-90}{$\prec$}} (-0.2,-.9);
\draw (1.2,-.1) --node[right]{\rotatebox{-90}{$\prec$}} (1.2,-.9);
\draw (0.1,0.2) --node[above]{$\prec'$} (.9,0.2);
\draw (.1,-1.2) --node[below]{$\prec'$} (.9,-1.2);
\node at (-0.2,0.2) {$\mu$};
\node at (1.2,0.2) {$\rho$};
\node at (-0.2,-1.15) {$\la$};
\node at (1.2,-1.2) {$\nu$};
\node at (.5,-.5) {$b$};
\end{tikzpicture}
\quad \text{with} \quad b \in \{0,1\}, \quad \nu \in \mc{U}^*(\la,\rho,k), \quad \mu \in \mc{D}^*(\la,\rho,k-b),
\end{equation}
where $k=|\nu/(\la\cup\rho)|=|(\la\cap \rho)/\mu|+b$.

\begin{defi}
A set of \deff{local growth rules} $\F_{\bullet}$ is a family of bijections
\[
\F_{\la,\rho,k}: \bigcup_{i=1}^k \mc{D}(\la,\rho,i) \rightarrow \mc{U}(\la,\rho,k),
\]
for each ordered pair $\la,\rho$ and non-negative integer $k$. Analogously, a set of \deff{local dual growth rules} $\F_{\bullet}^*$ is a family of bijections
\[
\F_{\la,\rho,k}^*:  \mc{D}^*(\la,\rho,k) \cup \mc{D}^*(\la,\rho,k-1) \rightarrow \mc{U}^*(\la,\rho,k),
\]
for each ordered pair $\la,\rho$ and non-negative integer $k$.
A (dual) growth $\Lambda$ is called an \deff{$\F_\bullet$-growth diagram} if each square of $\Lambda$ satisfies $\nu = \F_{\la,\rho,|(\la\cap\rho)/\mu|+a}(\mu)$ and an \deff{$\F_\bullet^*$-dual growth diagram} if each square of $\Lambda$ satisfies $\nu = \F^*_{\la,\rho,|(\la\cap\rho)/\mu|+b}(\mu)$,
where $\mu,\la,\rho,\nu,a,b$ are defined as in \eqref{eq:growth square} or \eqref{eq:dual growth square} respectively.
\end{defi}

The key observation is that each set of local (dual) growth rules yields a bijection between $n \times m$ matrices $A$ with non-negative entries and pairs $(P,Q)$ of SSYT of the same shape, or $n \times m$ $\{0,1\}$-matrices $B$ and pairs $(P,Q)$ of same shape, where $P$ is an SSYT and $Q$ a dual SSYT respectively, and hence a proof of Theorem~\ref{thm:Cauchy identities}. Indeed, starting with a matrix $A$ or $B$, we construct a (dual) growth diagram $\Lambda$ using the local (dual) growth rules and then map $\Lambda$ to the pair $(P(\Lambda),Q(\Lambda))$. On the other hand, starting with the pair $(P,Q)$ of tableaux we can construct a unique (dual) growth diagram by placing the chains of partitions corresponding to $P$ or $Q$ in the rightmost column or bottommost row and filling the remaining grid (including the matrix entries) according to
\[
\mu = \F_{\la,\rho,|\nu/(\la\cup\rho)|}^{-1}(\nu), \qquad a= |\nu|+|\mu|-|\la|-|\rho|,
\]
for the case of a growth or
\[
\mu = {\F_{\la,\rho,|\nu/(\la\cup\rho)|}^{*}}^{-1}(\nu), \qquad b= |\nu|+|\mu|-|\la|-|\rho|,
\]
for the case of a dual growth.

\begin{ex}
The $\F^\row$-growth diagram associated to the matrix
\[
A = \begin{pmatrix}
0 & 2 & 1 \\
1 & 1 & 0 \\
2 & 0 & 0
\end{pmatrix}
\]
is shown next. As in Example~\ref{ex:growths} we omit the zero entries in the growth diagram.
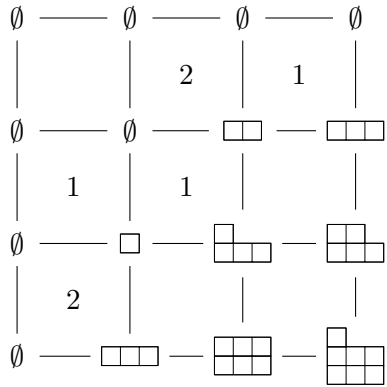
\begin{center}
\begin{tikzpicture}[scale=1.5]
\Yboxdim{.25cm}
\node at (0,-3) {$\emptyset$};
\node at (0,-2) {$\emptyset$};
\node at (0,-1) {$\emptyset$};
\node at (0,0) {$\emptyset$};
\node at (1,0) {$\emptyset$};
\node at (1,-1) {$\emptyset$};
\node at (2,0) {$\emptyset$};
\node at (3,0) {$\emptyset$};

\node at (1.5,-.5) {$2$};
\node at (2.5,-.5) {$1$};
\node at (.5,-1.5) {$1$};
\node at (1.5,-1.5) {$1$};
\node at (.5,-2.5) {$2$};

\node at (2,-1) {\yng(2)};
\node at (3,-1) {\yng(3)};
\node at (1,-2) {\yng(1)};
\node at (2,-2) {\yng(3,1)};
\node at (3,-2) {\yng(3,2)};
\node at (1,-3) {\yng(3)};
\node at (2,-3) {\yng(3,3)};
\node at (3,-3) {\yng(3,3,1)};

\draw (.2,0) -- (.8,0);
\draw (.2,-1) -- (.8,-1);
\draw (.2,-2) -- (.8,-2);
\draw (.2,-3) -- (.65,-3);
\draw (1.2,0) -- (1.8,0);
\draw (1.2,-1) -- (1.7,-1);
\draw (1.2,-2) -- (1.65,-2);
\draw (1.35,-3) -- (1.65,-3);
\draw (2.2,0) -- (2.8,0);
\draw (2.3,-1) -- (2.65,-1);
\draw (2.35,-2) -- (2.65,-2);
\draw (2.35,-3) -- (2.65,-3);

\draw (0,-.2) -- (0,-.8);
\draw (0,-1.2) -- (0,-1.8);
\draw (0,-2.2) -- (0,-2.8);
\draw (1,-.2) -- (1,-.8);
\draw (1,-1.2) -- (1,-1.8);
\draw (1,-2.2) -- (1,-2.8);
\draw (2,-.2) -- (2,-.8);
\draw (2,-1.2) -- (2,-1.7);
\draw (2,-2.3) -- (2,-2.7);
\draw (3,-.2) -- (3,-.8);
\draw (3,-1.2) -- (3,-1.7);
\draw (3,-2.3) -- (3,-2.65);
\end{tikzpicture}
\end{center}
By reading the last column and row of the above growth diagram we obtain
\[
P = \young(111,223,3)\,, \qquad \quad Q= \young(111,222,3)\,.
\]
\end{ex}

We denote by $\RSK_{\F_\bullet}$ and  $\RSK^*_{\F_\bullet^*}$ the bijections $A \mapsto (P,Q)$ induced by the local (dual) growth rules $\F_\bullet$ and $\F_\bullet^*$ as described above. In the remainder of this subsection we describe how $\RSK_{\F_\bullet}$ can be translated into an insertion algorithm. The translation of $\RSK^*_{\F_\bullet^*}$ into an insertion algorithm can be done analogously; for a detailed description see \cite[Lemma 2.10]{FriedenSchreierAigner24plus}. 
Since an insertion algorithm describes how the $P$ tableau, also called the \deff{insertion tableau}, changes, it suffices to restrict ourselves to one column of the $\F_\bullet$-growth diagram as displayed next. Remember that $T^{(i)}$ denotes the shape of a tableau when restricted to entries at most $i$.

\begin{center}
\begin{tikzpicture}[scale=1.75]
\node at (0,0) {$T^{(0)}$};
\node at (0,-1) {$T^{(1)}$};
\node at (0,-2) {$T^{(n-1)}$};
\node at (0,-3) {$T^{(n)}$};
\node at (1,0) {$\wh{T}^{(0)}$};
\node at (1,-1) {$\wh{T}^{(1)}$};
\node at (1,-2) {$\wh{T}^{(n-1)}$};
\node at (1,-3) {$\wh{T}^{(n)}$};
\node at (.5,-.5) {$A_{1,j}$};
\node at (.5,-2.5) {$A_{n,j}$};

\draw (0,-.2) -- (0,-.8);
\draw (1,-.2) -- (1,-.8);
\draw (0,-2.2) -- (0,-2.8);
\draw (1,-2.2) -- (1,-2.8);
\draw (.25,0) -- (.75,0);
\draw (.25,-1) -- (.75,-1);
\draw (.35,-2) -- (.65,-2);
\draw (.3,-3) -- (.7,-3);
\draw[very thick, dotted] (.5,-1.2) -- (.5,-1.8);
\end{tikzpicture}
\end{center}

\begin{lem}
\label{lem:insertion algorithm}
Given an SSYT $T$ and a multiset $I=\{1^{(A_{1,j})},2^{(A_{2,j})},\ldots,n^{(A_{n,j})}\}$ which we call the \deff{insertion set}. The tableau $\wh{T}$ as defined in the above partial $\F_\bullet$-growth diagram is obtained by the following algorithm.
\begin{itemize}
\item Let $i$ be the smallest element in $I$ and denote by $k$ its multiplicity. Denote by $C$ the set of cells of $\F_{\la,\rho,k}(\mu)/(\la \cup \rho)$ where $\mu=T^{(i-1)}$, $\la=T^{(i)}$ and $\rho=\wh{T}^{(i-1)}$. Place an entry $i$ into all cells of $C$, delete all $i$'s from the insertion set and add all entries to the insertion set which have been bumped, i.e., replaced by an entry $i$.
\item Repeat the previous step until the insertion set is empty.
\end{itemize}
\end{lem}
\begin{proof}
Let $i$ be the smallest element in $I$. The positions where an entry $i$ of $T$ was bumped in a previous step are the cells of $(\la \cap \rho)\setminus \mu$, where $\la,\rho,\mu$ are defined as above. By definition, the multiplicity $k$ of $i$ in $I$ is $A_{i,j}$ plus the number of times an entry $i$ has been bumped in a previous step, i.e., $k = A_{i,j}+ |(\la \cap \rho)\setminus \mu|$. The definition of $\nu=\F_{\la,\rho,k}(\mu)$ implies $k = |\nu \setminus (\la \cup \rho)|$. Hence we can indeed obtain $\nu$ by placing all $i$ entries of $C$ into the shape $\nu \setminus (\la \cup \rho)$.
\end{proof}

For a set of dual local growth rules $\F^*_\bullet$ there is an analogous algorithm by restricting to sets instead of multisets; compare to \cite[Lemma 2.10]{FriedenSchreierAigner24plus} for more details.
We define the \deff{$\F_\bullet$-insertion} of $I$ into $T$ as the algorithm described in Lemma~\ref{lem:insertion algorithm} and denote it by $ T \xleftarrow[\F_\bullet]{}I$. Analogously we define the \deff{dual $\F^*_\bullet$-insertion} by restricting the above algorithm to sets.
We call a (dual) $\F_\bullet$-insertion \deff{traceable} if 
\[
 T \xleftarrow[\F_\bullet]{}I =
 T \xleftarrow[\F_\bullet]{}\{i_1\} \xleftarrow[\F_\bullet]{}\{i_2\}  \xleftarrow[\F_\bullet]{} \cdots \xleftarrow[\F_\bullet]{}\{i_k\},
\]
holds for any SSYT $T$ and set $I=\{i_1 \leq \cdots \leq i_k\}$, i.e., we are allowed to insert the elements of $I$ ``sequentially''. Analogously we call a (dual) $\F_\bullet$-insertion \deff{reverse traceable} if 
\[
 T \xleftarrow[\F_\bullet]{}I =
 T \xleftarrow[\F_\bullet]{}\{i_k\} \xleftarrow[\F_\bullet]{}\{i_{k-1}\}  \xleftarrow[\F_\bullet]{} \cdots \xleftarrow[\F_\bullet]{}\{i_1\},
\]
holds for any SSYT $T$ and any set $I=\{i_1 \leq \cdots \leq i_k\}$.
It is easy to verify, that the row insertion $\F^{\row}_\bullet$ and the dual column insertion $\F^{\col^*}_\bullet$ are traceable and that the column insertion $\F^{\col}_\bullet$  and the dual row insertion $\F^{\row^*}_\bullet$ are reverse traceable; compare for example with \cite[Section 3]{vanLeeuwen05}.\bigskip

Given an $n \times m$ matrix $A$ with non-negative entries, we can construct the image of $\RSK_{\F_\bullet}$ using the $\F_\bullet$-insertion  as follows. First we construct the biword $\omega_A$ by reading the columns of $A$ from top to bottom, starting with the leftmost column, and appending $A_{i,j}$-times $\begin{pmatrix}
j \\ i \end{pmatrix}$ to $\omega_A$.
Given a biword
\[
\omega_A = \begin{pmatrix}
j_1 & \cdots &j_1 & \cdots& \cdots & j_k & \cdots & j_k\\
i_{1,1} & \cdots & i_{1,l_1} & \cdots & \cdots & i_{k,1} & \cdots & i_{k,l_k}
\end{pmatrix},
\]
the \deff{insertion tableau} $P$ is defined as
\[
P= \emptyset \xleftarrow[\F_\bullet]{}\{i_{1,1},\ldots,i_{1,l_1}\} \xleftarrow[\F_\bullet]{} \cdots \xleftarrow[\F_\bullet]{} \{i_{k,1},\ldots,i_{k,l_k}\}.
\]
The \deff{recording tableau} $Q$ is defined as the SSYT which records the growth of $P$ by placing $j_r$ into the cells which were added during the $\F_\bullet$-insertion of $\{i_{r,1},\ldots,i_{r,l_r}\}$.

\begin{ex}
The biword $\omega_A$ of the matrix 
\[
A = \begin{pmatrix}
0 & 2 & 1 \\
1 & 1 & 0 \\
2 & 0 & 0
\end{pmatrix}
\]
as in Example~\ref{ex:growths} is
\[
\omega_A = \begin{pmatrix}
1 & 1 & 1 & 2 & 2 & 2 & 3 \\
2 & 3 & 3 & 1 & 1 & 2 & 1
\end{pmatrix}.
\]
Using that $\F^\row_\bullet$ is traceable, the stepwise construction of the insertion tableau $P$ and the recording tableau $Q$ of $\RSK_{\F_\bullet^\row}$ is shown next.
\[
\begin{array}{ccccccccc}
\emptyset & \young(2) & \young(23) & \young(233) & \young(133,2) & \young(113,23) & \young(112,233) & \young(111,223,3) &=P \\[24pt]
\emptyset & \young(1) & \young(11) & \young(111) & \young(111,2) & \young(111,22) & \young(111,222) & \young(111,222,3) &=Q
\end{array}
\]
By comparing with Example~\ref{ex:growths}, we see that we obtain the same pair of tableaux.
\end{ex}

\begin{rem}
\label{rem:sym of RSK}
We call a set of local growth rules $\F_\bullet$ \deff{symmetric}, if $\F_{\la,\rho,k}(\mu)=\F_{\rho,\la,k}(\mu)$ for all partitions $\la,\rho,\mu$ and integers $k$. It is immediate that the $\RSK_{\F_\bullet}$ correspondence of a symmetric set of local growth rules $\F_\bullet$ satisfies the following property
\begin{equation}
\label{eq:sym of RSK}
\RSK_{\F_\bullet}(A) =(P,Q) \qquad \Leftrightarrow \qquad \RSK_{\F_\bullet}(A^T) =(Q,P).
\end{equation}
It is not difficult to see that the sets $\mc{D}(\la,\rho,k)$ and $\mc{U}(\la,\rho,k)$ can be expressed by using $\la \cap \rho$ and $\la \cup \rho$ instead of $\la,\rho$. Hence all local growth rules whose definition is based on addable and removable ribbons, in particular $\F_\bullet^{\row}$ and $\F_{\bullet}^\col$, are symmetric.
\end{rem}

\subsection{Skew identities}
\label{sec:skew}
One of the advantages of using growth diagrams is that we can generalise the above results immediately to the setting of \emph{skew shapes} without any extra work. \\

Let $\mu \subseteq \la$ be two partitions. A filling of the cells of $\la/\mu$ with positive integers is called
\begin{itemize}
\item a \deff{semistandard Young tableau (SSYT)} of skew shape $\la/\mu$ if the entries are weakly increasing along rows and strictly increasing along columns,
\item a \deff{dual semistandard Young tableau} of skew shape $\la/\mu$ if the entries are strictly increasing along rows and weakly increasing along columns.
\end{itemize}
We denote by $\SSYT_{\la/\mu}(n)$ the set of semistandard Young tableaux of skew shape $\la/\mu$ with entries at most $n$ and by $\SSYT_{\la/\mu}^*(n)$ the set of dual semistandard Young tableaux of skew shape $\la/\mu$ and entries at most $n$. The \deff{Schur polynomial} $s_{\la/\mu}(\x)$ is defined for a family of variables $\x=(x_1,\ldots,x_n)$ as
\[
s_{\la/\mu}(\x) = \sum_{T \in \SSYT_{\la/\mu}(n)}\x^T,
\]
where the weight $\x^T$ of a skew shaped tableau $T$ is defined analogously to the regular case.

\begin{thm}
\label{thm:skew Cauchy identities}
Let $\x=(x_1,\ldots, x_n)$ and $\y=(y_1,\ldots,y_m)$ be two sequences of variables and $\la,\rho$ be two partitions. Then
\begin{align}
\label{eq:skew Cauchy}
\sum_{\nu} s_{\nu/\rho}(\x) s_{\nu/\la}(\y) &= \prod_{\substack{1 \leq i \leq n \\ 1 \leq j \leq m}} \frac{1}{1-x_iy_j} \sum_{\mu} s_{\la/\mu}(\x) s_{\rho/\mu}(\y),\\
\label{eq:skew dual Cauchy}
\sum_{\nu} s_{\nu/\rho}(\x) s_{\nu^\prime/\la^\prime}(\y) &= \prod_{\substack{1 \leq i \leq n \\ 1 \leq j \leq m}} (1+x_iy_j) \sum_{\mu} s_{\la/\mu}(\x) s_{\rho^\prime/\mu^\prime}(\y).
\end{align}
\end{thm}

Algebraically, both identities follow immediately from the commutation relations in Theorem~\ref{thm:commutation relations} as described in the paragraph after \eqref{eq:cauchy operator inner product} by using the fact, that the Schur polynomials for skew shapes can be expressed as
\begin{align}
\label{eq:skew schur as ups}
s_{\la/\mu}(\x) &= \left\langle U_{x_n} \cdots U_{x_1} \mu , \la \right\rangle, \\
\label{eq:skew schur as downs}
s_{\la/\mu}(\y) &= \left\langle D_{y_1} \cdots D_{y_m} \la , \mu \right\rangle, \\
\label{eq:skew schur as dual downs}
s_{\la^\prime/\mu^\prime}(\y) &= \left\langle D_{y_1}^* \cdots D_{y_m}^* \la , \mu \right\rangle.
\end{align}
On the other side, the skew Cauchy identities in Theorem~\ref{thm:skew Cauchy identities} actually imply the commutation relations in Theorem~\ref{thm:commutation relations}, i.e., both theorems are equivalent to each other. Indeed, by restricting to one variable in both families, i.e., $\x=(x), \y=(y)$, the left hand side of \eqref{eq:skew Cauchy} is equal to
\[
\sum_{\nu} s_{\nu/\rho}(x) s_{\nu/\la}(y)
 = \sum_{\nu} \left\langle U_x \rho,\nu\right\rangle \left\langle D_y \nu,\la\right\rangle 
 = \left\langle D_y U_x \rho , \la\right\rangle,
\]
and the right hand side of \eqref{eq:skew Cauchy} is equal to
\[
 \frac{1}{1-xy} \sum_{\mu} s_{\la/\mu}(x) s_{\rho/\mu}(y)
 = \frac{1}{1-xy} \sum_{\mu} \left\langle U_x \mu , \la \right\rangle \left\langle D_y \rho , \mu \right\rangle 
 = \left\langle \frac{1}{1-xy} U_x D_y \rho , \la\right\rangle,
\]
which implies \eqref{eq:commutation relation}. The commutation relation \eqref{eq:dual commutation relation} follows analogously from \eqref{eq:skew dual Cauchy}.\\

The immediate generalisation of the algebraic proof to the skew shape setting can also be transferred to the combinatorial level. We present this explicitly for \eqref{eq:skew Cauchy}, the case \eqref{eq:skew dual Cauchy} is obtained analogously. First we interpret the left hand side of \eqref{eq:skew Cauchy} as the generating function of pairs $(P,Q)$ of semistandard Young tableaux where $P\in \SSYT_{\nu/\rho}(n)$ and $Q \in \SSYT_{\nu/\la}(m)$ with respect to the weight $\x^P\y^Q$, and the right hand side as the generating function of triples $(A,S,T)$ where $A$ is an $n\times m$ matrix with non-negative integer entries, $S \in \SSYT_{\la/\mu}(n)$, and $T \in \SSYT_{\rho/\mu}(m)$ with respect to the weight
\[
 \x^{S} \y^{T}
\prod_{\substack{1 \leq i \leq n \\ 1\leq j \leq m}}(x_i y_j)^{a_{i,j}}.
\]
For a given set of local growth rules $\F_\bullet$ we can easily extend $\RSK_{\F_\bullet}$ to a weight preserving bijection between such triples $(A,S,T)$ and pairs $(P,Q)$. For a triple $(A,S,T)$, we construct an $\F_\bullet$-growth diagram $\Lambda$ associated to $(A,S,T)$ by setting $\Lambda_{i,0}=S^{(i)}$, $\Lambda_{0,j}=T^{(j)}$ and using the local growth rules $\F_\bullet$  to fill the rest of $\Lambda$. We define $P=P(\Lambda)$ and $Q=Q(\Lambda)$ as before.
Note that the resulting growth $\Lambda$ is actually a generalisation of the objects defined in Definition~\ref{def:growth} by replacing the condition \eqref{eq:growth sizes} with 
\[
|\Lambda_{i,j}| =|S^{(i)}|+|T^{(j)}|-|S^{(0)}|+ \sum_{\substack{1 \leq k \leq i \\ 1 \leq l \leq j}}A_{k,l}.
\]
For simplicity we refer to both objects as \emph{growths}.
Since the local growth rules $F_\bullet$ are bijections, it is immediate that $\RSK_{\F_\bullet}$ proves \eqref{eq:skew Cauchy} bijectively.

\begin{ex}
The $\F^\row$-growth diagram associated to the triple $(A,S,T)$ for
\[
A = \begin{pmatrix}
1 & 0 & 0 \\
0 & 0 & 2 \\
0 & 1 & 0
\end{pmatrix}, \qquad
S= \begin{tikzpicture}[baseline=.4cm]
\tyoung(0cm,0cm,::2,23)
\Ygray
\tgyoung(0cm,0cm,;;)
\end{tikzpicture}\,, \qquad
T= \begin{tikzpicture}[baseline=.4cm]
\tyoung(0cm,0cm,::13,2)
\Ygray
\tgyoung(0cm,0cm,;;)
\end{tikzpicture}\,,
\]
where the cells of $\mu$ for $S$ and $T$ are coloured grey, is shown next.
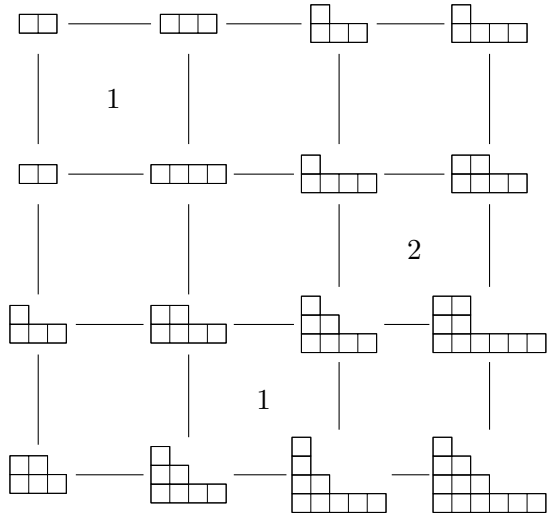
\begin{center}
\begin{tikzpicture}[scale=2]
\Yboxdim{.25cm}
\node at (.5,-.5) {$1$};
\node at (1.5,-2.5) {$1$};
\node at (2.5,-1.5) {$2$};

\node at (0,0) {\yng(2)};
\node at (1,0) {\yng(3)};
\node at (2,0) {\yng(3,1)};
\node at (3,0) {\yng(4,1)};

\node at (0,-1) {\yng(2)};
\node at (1,-1) {\yng(4)};
\node at (2,-1) {\yng(4,1)};
\node at (3,-1) {\yng(4,2)};

\node at (0,-2) {\yng(3,1)};
\node at (1,-2) {\yng(4,2)};
\node at (2,-2) {\yng(4,2,1)};
\node at (3,-2) {\yng(6,2,2)};

\node at (0,-3) {\yng(3,2)};
\node at (1,-3) {\yng(4,2,1)};
\node at (2,-3) {\yng(5,2,1,1)};
\node at (3,-3) {\yng(6,3,2,1)};

\draw (.2,0) -- (.75,0);
\draw (.2,-1) -- (.7,-1);
\draw (.25,-2) -- (.7,-2);
\draw (.25,-3) -- (.7,-3);
\draw (1.25,0) -- (1.75,0);
\draw (1.3,-1) -- (1.7,-1);
\draw (1.3,-2) -- (1.7,-2);
\draw (1.3,-3) -- (1.65,-3);
\draw (2.25,0) -- (2.7,0);
\draw (2.3,-1) -- (2.7,-1);
\draw (2.3,-2) -- (2.6,-2);
\draw (2.35,-3) -- (2.6,-3);

\draw (0,-.2) -- (0,-.8);
\draw (1,-.2) -- (1,-.8);
\draw (2,-.2) -- (2,-.8);
\draw (3,-.2) -- (3,-.8);
\draw (0,-1.2) -- (0,-1.8);
\draw (1,-1.2) -- (1,-1.8);
\draw (2,-1.2) -- (2,-1.75);
\draw (3,-1.2) -- (3,-1.75);

\draw (0,-2.2) -- (0,-2.8);
\draw (1,-2.2) -- (1,-2.75);
\draw (2,-2.25) -- (2,-2.7);
\draw (3,-2.25) -- (3,-2.7);
\end{tikzpicture}
\end{center}
By reading the last column and row of the above growth diagram we obtain
\[
P = \begin{tikzpicture}[baseline=.9cm]
\tyoung(0cm,0cm,::::22,:13,22,3)
\Ygray
\tgyoung(0cm,0cm,;;;;,;)
\end{tikzpicture}\,, \qquad \quad Q= \begin{tikzpicture}[baseline=.9cm]
\tyoung(0cm,0cm,:::123,::3,13,2)
\Ygray
\tgyoung(0cm,0cm,;;;,;;)
\end{tikzpicture}\,.
\]
\end{ex}

As another consequence we obtain combinatorial proofs for the (dual) Pieri rules. Indeed, the skew (dual) Cauchy identity for $\rho=\mu=\emptyset$ and $\y=(y_1)$ implies
\begin{align}
\label{eq:Pieri via Cauchy}
\sum_{k \geq 0} y_1^k h_k(\x) s_\la(\x) = \prod_{i=1}^n \frac{1}{1-x_iy_1} s_\la(\x)
= \sum_{\nu} s_\nu(\x) s_{\nu/\la}(y_1) = \sum_{\nu \succ \la} y_1^{|\nu/\la|}s_\la(\x),\\
\label{eq:dual Pieri via dual Cauchy}
\sum_{k = 0}^n y_1^k e_k(\x) s_\la(\x) = \prod_{i=1}^n (1+x_iy_1) s_\la(\x)
= \sum_{\nu} s_\nu(\x) s_{\nu^\prime/\la^\prime}(y_1) = \sum_{\nu \succ^\prime \la} y_1^{|\nu/\la|}s_\la(\x),
\end{align}
where we used the generating functions for $h_k(\x)$ and $e_k(\x)$ on the one hand and the fact that $s_{\nu/\la}(y_1)$ is $0$ unless $\nu/\la$ is a horizontal strip. By comparing the coefficient of $y_1^k$ on both sides of \eqref{eq:Pieri via Cauchy} and \eqref{eq:dual Pieri via dual Cauchy} respectively, we obtain the Pieri and dual Pieri rule
\begin{equation}
\label{eq:Pieri and dual Pieri}
h_k(\x)s_\la(\x) = \sum_{\substack{\nu \succ \la \\ |\nu/\la|=k}}s_\nu(\la), \qquad
e_k(\x)s_\la(\x) = \sum_{\substack{\nu \succ^\prime \la \\ |\nu/\la|=k}}s_\nu(\la).
\end{equation}
For each choice of local (dual) growth rules $\F_\bullet$ we obtain a bijection between all pairs $(T,(a_i)_i)$ where $T \in \SSYT_\la(n)$ and $(a_i)_{1 \leq i \leq n}$ is a sequence of non-negative integers (resp., $a_i \in \{0,1\}$) with sum equal to $k$ and all $\wh{T} \in \SSYT_\nu(n)$ where $\nu \succ \la$ (resp., $\nu \succ^\prime \la$) and $|\nu/\la|=k$ by $\F_\bullet$-inserting (resp., $\F_\bullet^*$-inserting) $\{1^{(a_1)},\ldots,n^{(a_n)}\}$ into $T$.

\section{Classical Littlewood identities}
\label{sec:drei}

\subsection{Projection identities}
\label{sec:projection identities}
In this section we introduce a new type of identity relating the up and (dual) down operators which we call \deff{projection identities}. While the commutation relations of the previous section allow us to interchange the operators, these new identities allow us to ``replace'' up operators by (dual) down operators.\\

For a set $X$ of partitions we define the $\bb{Q}$-linear map $\pi_X: \bb{Y} \rightarrow \bb{Q}$ as
\[
\pi_X(\mu) := \sum_{\la \in X} \left\langle \mu,\la \right\rangle.
\]
Note that the linear extension of $\pi_X$ to $\bb{Y}\llbracket x_1,\ldots,x_n \rrbracket$ is well-defined even for an infinite set $X$ since the coefficients of the monomials of an element in $\bb{Y}\llbracket x_1,\ldots,x_n \rrbracket$ are always finite formal sums of partitions.
Following the convention of Ayyer and Kumari \cite[Definition 2.9]{AyyerKumari22} we call a partition \deff{$k$-asymmetric} if it has Frobenius notation $(a_1,\ldots,a_l|a_1+k,\ldots,a_l+k)$ for a non-negative integer $k$ and $(a_1-k,\ldots,a_l-k|a_1,\ldots,a_l)$ for a negative integer $k$. Denote by $\bb{P}_e$ the set of partitions with even parts, $\bb{P}_e^\prime$ the set of partitions with even columns, and by $\bb{P}_{\asym}^k$ the set of $k$-asymmetric partitions.

\begin{thm}
\label{thm:projection identities}
The up and (dual) down operators satisfy the following \deff{projection identities}
\begin{align}
\label{eq:projection all partitions}
\pi_{\bb{P}} U_x &= \frac{1}{1-x} \pi_{\bb{P}} D_x, \\
\label{eq:projection even partitions}
\pi_{\bb{P}_e} U_x &= \frac{1}{1-x^2} \pi_{\bb{P}_e} D_x, \\
\label{eq:projection even col partitions}
\pi_{\bb{P}_e^\prime} U_x &= \pi_{\bb{P}_e^\prime} D_x, \\
\label{eq:projection 1asym partitions}
\pi_{\bb{P}_\asym^{1}} U_x &= \pi_{\bb{P}_\asym^{1}} D_x^*, \\
\label{eq:projection -1asym partitions}
\pi_{\bb{P}_\asym^{-1}} U_x &= (1+x^2) \pi_{\bb{P}_\asym^{-1}} D_x^*.
\end{align}
\end{thm}

Before proving the above theorem we show how it yields an algebraic proof of the classical Littlewood identity \eqref{eq:littlewood identity} for $\x=(x_1,\ldots,x_n)$. The other Littlewood identities follow analogously.
Using \eqref{eq:schur as ups}, we can rewrite the left hand side of \eqref{eq:littlewood identity} as
\[
\sum_{\la \in \bb{P}} s_\la(\x) = 
\sum_{\la \in \bb{P}} \left\langle U_{x_n}\cdots U_{x_1} \emptyset, \la \right\rangle 
= \pi_{\bb{P}} U_{x_n}\cdots U_{x_1} \emptyset.
\]
By the commutation relation \eqref{eq:commutation relation} and  the projection identity \eqref{eq:projection all partitions} we can further rewrite the above as
\begin{multline*}
\pi_{\bb{P}} U_{x_n}\cdots U_{x_1} \emptyset =
\frac{1}{1-x_n}\pi_{\bb{P}} D_{x_n} U_{x_{m-1}}\cdots U_{x_1} \emptyset 
\\=
\frac{1}{1-x_n} \cdot \frac{1}{1-x_{n-1}x_n}\cdot \pi_{\bb{P}} U_{x_{n-1}} D_{x_n} U_{x_{n-2}}\cdots U_{x_1} \emptyset 
=  \cdots  \\=
\frac{1}{1-x_n} \prod_{i=1}^{n-1}\frac{1}{1-x_ix_n} \cdot \pi_{\bb{P}} U_{x_{n-1}}\cdots U_{x_1} D_{x_n}  \emptyset 
\\ =
\frac{1}{1-x_n} \prod_{i=1}^{n-1}\frac{1}{1-x_ix_n} \cdot \pi_{\bb{P}} U_{x_{n-1}}\cdots U_{x_1} \emptyset
=  \cdots  
=
\prod_{i=1}^n \frac{1}{1-x_i} \prod_{1 \leq i < j \leq n} \frac{1}{1-x_i x_j},
\end{multline*}
where we used $D_x \emptyset = \emptyset$.\bigskip

As for the commutation relations, we prove Theorem~\ref{thm:projection identities} bijectively. In particular we relate the equations \eqref{eq:projection all partitions} and \eqref{eq:projection even partitions} to special cases of \eqref{eq:commutation relation}, and provide explicit bijections for the equations \eqref{eq:projection 1asym partitions} and \eqref{eq:projection -1asym partitions}. The equation \eqref{eq:projection even col partitions} on the other hand turns out to be ``trivial''.
For a partition $\la$, an integer $k$ and a subset $X$ of partitions we define the sets
\begin{align*}
\mc{U}_X(\la,k) &:= \{ \nu \in X: \nu \succ \la, |\nu /\la| = k\}, \\
\mc{D}_X(\la,k) &:= \{ \mu \in X: \mu \prec \la, |\la/\mu| = k\}, \\
\mc{D}_X^*(\la,k) &:= \{ \mu \in X: \mu \prec^\prime \la, |\la/\mu| = k\}.
\end{align*}
The projection identities \eqref{eq:projection all partitions}, \eqref{eq:projection even partitions}, and \eqref{eq:projection even col partitions} are equivalent to the respective system of equations
\[
\left| \mc{U}_{\bb{P}}(\la,k) \right| = \sum_{i=0}^k \left| \mc{D}_{\bb{P}}(\la,i) \right|, \qquad
\left| \mc{U}_{\bb{P}_e}(\la,k) \right| = \sum_{i=0}^{\left\lfloor \frac{k}{2}\right\rfloor} \left| \mc{D}_{\bb{P}_e}(\la,k-2i) \right|, \qquad
\left| \mc{U}_{\bb{P}_e^\prime}(\la,k) \right| = \left| \mc{D}_{\bb{P}_e^\prime}(\la,k) \right|,
\]
and the identities \eqref{eq:projection 1asym partitions} and  \eqref{eq:projection -1asym partitions} are equivalent to the respective system of equations
\[
\left| \mc{U}_{\bb{P}_\asym^{1}}(\la,k) \right| =\left| \mc{D}_{\bb{P}_\asym^{1}}(\la,k) \right|, \qquad\qquad
\left| \mc{U}_{\bb{P}_\asym^{-1}}(\la,k) \right| = \left| \mc{D}_{\bb{P}_\asym^{-1}}(\la,k) \right| + \left| \mc{D}_{\bb{P}_\asym^{-1}}(\la,k-2) \right|,
\]
for all partitions $\la$ and non-negative integers $k$.

\subsubsection{Proof of \eqref{eq:projection all partitions}}
\label{sec:proof all partitions}
Let $\F_\bullet$ be a set  of local growth rules. It is immediate that $\mc{D}_\bb{P}(\la,k)=\mc{D}(\la,\la,k)$ and $\mc{U}_\bb{P}(\la,k)=\mc{U}(\la,\la,k)$ holds for all partitions $\la$ and integers $k$. Hence the map $\F_{\bb{P},\la,k}$
\begin{equation}
\label{eq:natural P rule}
\F_{\bb{P},\la,k}: \bigcup_{i \leq k} \mc{D}_\bb{P}(\la,i) \rightarrow \mc{U}_\bb{P}(\la,k):\quad
\mu \mapsto \F_{\la,\la,k}(\mu).
\end{equation}
is well defined and a bijection, which  yields a bijective proof of \eqref{eq:projection all partitions}.

\subsubsection{Proof of \eqref{eq:projection even partitions}}
\label{sec:proof even partitions}
For a partition $\la=(\la_1,\ldots, \la_n)$ define the partitions $\la^-,\la^+$ as
\[
\la^-=\left(\left\lfloor \frac{\la_1}{2} \right\rfloor, \ldots, \left\lfloor \frac{\la_n}{2} \right\rfloor\right),
\qquad
\la^+=\left(\left\lceil \frac{\la_1}{2} \right\rceil, \ldots, \left\lceil \frac{\la_n}{2} \right\rceil\right),
\]
and write $\odd(\la)$ for the number of odd parts of $\la$.
One can verify easily that the maps
\begin{align*}
\mc{D}(\la^-,\la^+,k) \rightarrow \mc{D}_{\bb{P}_e}(\la,2k+\odd(\la)): \quad
(\mu_1,\ldots,\mu_n) \mapsto (2\mu_1,\ldots,2\mu_n),\\
\mc{U}(\la^-,\la^+,k) \rightarrow \mc{U}_{\bb{P}_e}(\la,2k+\odd(\la)): \quad
(\nu_1,\ldots,\nu_n) \mapsto (2\nu_1,\ldots,2\nu_n),
\end{align*}
are bijections; compare with Figure~\ref{fig:even partitions bij}. By abuse of notation, we call both bijections $\varphi$.
For a given set of local growth rules $\F_\bullet$, we define the map $\F_{\bb{P}_e,\la,2k+\odd(\la)}$ as
\begin{equation}
\label{eq:natural Pe rule}
\F_{\bb{P}_e,\la,2k+\odd(\la)}= \varphi \circ \F_{\la^-,\la^+,k} \circ \varphi^{-1}. 
\end{equation}
It is immediate by the above considerations, that $\F_{\bb{P}_e,\la,2k+\odd(\la)}$ is a bijection between the sets $\bigcup_{i\leq k} \mc{D}_{\bb{P}_e}(\la,2i+\odd(\la))$ and $\mc{U}_{\bb{P}_e}(\la,2k+\odd(\la))$ and hence yields a bijective proof of \eqref{eq:projection even partitions}.


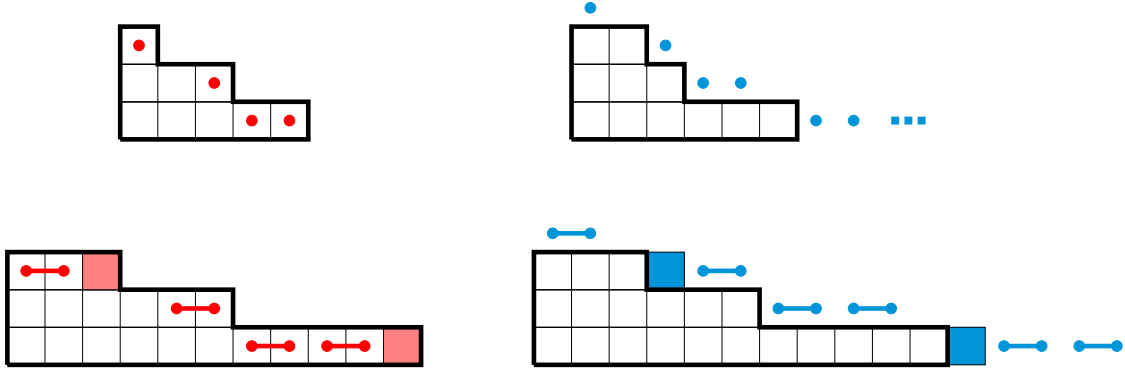
\begin{figure}[h]
\begin{center}
\begin{tikzpicture}

\begin{scope}[yshift=-3cm]
\tyng(0,0,11,6,3)
{\Yred
\tgyoung(0cm,0cm,::::::::::;,,::;)
}
\draw[line width=1.75pt] (0,0) -- (5.5,0) -- (5.5,.5) -- (3,.5) -- (3,1) -- (1.5,1) -- (1.5,1.5) -- (0,1.5) -- (0,0);
\draw[fill, red] (4.75,.25) circle (2pt);
\draw[fill, red] (4.25,.25) circle (2pt);
\draw[fill, red] (3.75,.25) circle (2pt);
\draw[fill, red] (3.25,.25) circle (2pt);
\draw[fill, red] (2.75,.75) circle (2pt);
\draw[fill, red] (2.25,.75) circle (2pt);
\draw[fill, red] (.75,1.25) circle (2pt);
\draw[fill, red] (.25,1.25) circle (2pt);
\draw[line width=1.75pt, red] (4.25,.25) -- (4.75,.25);
\draw[line width=1.75pt, red] (3.25,.25) -- (3.75,.25);
\draw[line width=1.75pt, red] (2.25,.75) -- (2.75,.75);
\draw[line width=1.75pt, red] (.25,1.25) -- (.75,1.25);

\begin{scope}[xshift=7cm]
\tyng(0,0,11,6,3)
{\Ycyan
\tgyoung(0cm,0cm,:::::::::::;,,:::;)
}
\draw[line width=1.75pt] (0,0) -- (5.5,0) -- (5.5,.5) -- (3,.5) -- (3,1) -- (1.5,1) -- (1.5,1.5) -- (0,1.5) -- (0,0);

\draw[fill, myblue] (7.25,.25) circle (2pt);
\draw[fill, myblue] (7.75,.25) circle (2pt);
\draw[fill, myblue] (6.25,.25) circle (2pt);
\draw[fill, myblue] (6.75,.25) circle (2pt);
\draw[fill, myblue] (4.75,.75) circle (2pt);
\draw[fill, myblue] (4.25,.75) circle (2pt);
\draw[fill, myblue] (3.75,.75) circle (2pt);
\draw[fill, myblue] (3.25,.75) circle (2pt);
\draw[fill, myblue] (2.75,1.25) circle (2pt);
\draw[fill, myblue] (2.25,1.25) circle (2pt);
\draw[fill, myblue] (.75,1.75) circle (2pt);
\draw[fill, myblue] (.25,1.75) circle (2pt);
\draw[line width=1.75pt, myblue] (7.25,.25) -- (7.75,.25);
\draw[line width=1.75pt, myblue] (6.25,.25) -- (6.75,.25);
\draw[line width=1.75pt, myblue] (4.25,.75) -- (4.75,.75);
\draw[line width=1.75pt, myblue] (3.25,.75) -- (3.75,.75);
\draw[line width=1.75pt, myblue] (2.25,1.25) -- (2.75,1.25);
\draw[line width=1.75pt, myblue] (.25,1.75) -- (.75,1.75);
\end{scope}
\end{scope}

\begin{scope}[xshift=1.5cm]
\tyng(0,0,5,3,1)
\draw[line width=1.75pt] (0,0) -- (2.5,0) -- (2.5,.5) -- (1.5,.5) -- (1.5,1) -- (.5,1) -- (.5,1.5) -- (0,1.5) -- (0,0);
\draw[fill, red] (2.25,.25) circle (2pt);
\draw[fill, red] (1.75,.25) circle (2pt);
\draw[fill, red] (1.25,.75) circle (2pt);
\draw[fill, red] (.25,1.25) circle (2pt);
\end{scope}

\begin{scope}[xshift=7.5cm]
\tyng(0,0,6,3,2)
\draw[line width=1.75pt] (0,0) -- (3,0) -- (3,.5) -- (1.5,.5) -- (1.5,1) -- (1,1) -- (1,1.5) -- (0,1.5) -- (0,0);

\draw[line width=3pt, dotted, myblue] (4.25,.25) -- (4.75,.25);

\draw[fill, myblue] (3.75,.25) circle (2pt);
\draw[fill, myblue] (3.25,.25) circle (2pt);
\draw[fill, myblue] (2.25,.75) circle (2pt);
\draw[fill, myblue] (1.75,.75) circle (2pt);
\draw[fill, myblue] (1.25,1.25) circle (2pt);
\draw[fill, myblue] (.25,1.75) circle (2pt);
\end{scope}
\end{tikzpicture}
\caption{\label{fig:even partitions bij} On the top row are the partitions $\la^-$ and $\la^+$ for $\la=(11,6,3)$, where the boxes lying in an inner (resp., outer) horizontal ribbon are marked by dots. The bottom row contains the corresponding images under $\varphi$, where the boxes which have to be removed (resp. added) to obtain $\la$ are filled in red (resp. blue) and single dots are mapped to pairs of dots.}
\end{center}
\end{figure}

\subsubsection{Proof of \eqref{eq:projection even col partitions}}
Let $\la$ be a partition. Then both sets $\mc{U}_{\bb{P}_e^\prime}(\la,k)$ and $\mc{D}_{\bb{P}_e^\prime}(\la,k)$ are empty unless $k = \odd(\la^\prime)$ in which case $\mc{U}_{\bb{P}_e^\prime}(\la,k)$ consists of the partition $\nu$ obtained from $\la$ by adding to each odd column a cell, and $\mc{D}_{\bb{P}_e^\prime}(\la,k)$ consists of the partition $\mu$ obtained from $\la$ by removing a cell of each odd column. We denote by $\F_{\bb{P}_e^\prime,\la,\odd(k)}$ the bijection from $\{\mu\}$ to $\{\nu\}$.

\subsubsection{Proof of \eqref{eq:projection 1asym partitions}}
\begin{figure}
\begin{center}
\begin{tikzpicture}

\tyng(0,0,8,7,6,6,5,5,3,2,2,1,1)
{\Yred
\tgyoung(0cm,0cm,,,,:::::;,,,,,:;,;,;)
}
{\Ygray
\tgyoung(0cm,0cm,;,:;,::;,:::;,::::;)
}
\node at (3.25,.75) {\color{red}{$2$}};
\node at (2.75,1.25) {\color{red}{$3$}};
\node at (2.25,2.25) {\color{red}{$5$}};
\node at (2.25,2.75) {\color{red}{$5$}};
\node at (1.25,3.25) {\color{red}{$3$}};
\node at (.75,3.75) {\color{red}{$2$}};
\draw[line width=1.75pt] (0,0) -- (4,0) -- (4,.5) -- (3.5,.5) -- (3.5,1) -- (3,1) -- (3,2) -- (2.5,2) -- (2.5,3) -- (1.5,3) -- (1.5,3.5) -- (1,3.5) -- (1,4.5) -- (.5,4.5) -- (.5,5.5) -- (0,5.5) -- (0,0);

\begin{scope}[xshift=7cm]
\tyng(0,0,8,7,6,6,5,5,3,2,2,1,1)
{\Ycyan
\tgyoung(0cm,0cm,::::::::;;,:::::::;,,,,,:::;)
}
{\Ygray
\tgyoung(0cm,0cm,;,:;,::;,:::;,::::;)
}
\node at (5.25,.25) {\color{myblue}{$1$}};
\node at (3.25,1.25) {\color{myblue}{$3$}};
\node at (2.75,2.25) {\color{myblue}{$5$}};
\node at (2.25,3.25) {\color{myblue}{$5$}};
\node at (1.25,3.75) {\color{myblue}{$3$}};
\node at (.25,5.75) {\color{myblue}{$1$}};
\draw[line width=1.75pt] (0,0) -- (4,0) -- (4,.5) -- (3.5,.5) -- (3.5,1) -- (3,1) -- (3,2) -- (2.5,2) -- (2.5,3) -- (1.5,3) -- (1.5,3.5) -- (1,3.5) -- (1,4.5) -- (.5,4.5) -- (.5,5.5) -- (0,5.5) -- (0,0);
\end{scope}
\end{tikzpicture}
\caption{\label{fig:ex for 1 asym} On the left a sketch of all $\mu \in \bb{P}^{1}_{\asym}$ with $\mu \prec^\prime \la=(7,5,3,2,0|10,7,4,2,1)$, where the boxes along the diagonal are marked gray, the boxes which have to be removed are in red and a box labelled $i$ is removed if $i \in \Ret^{1}_\la(\mu)$.
On the right a sketch of all $\nu \in \bb{P}^{1}_{\asym}$ with $\nu\succ \la$, where boxes are in blue if they have to be added and a label $i$ marks a box which is added if $i\in \Set^{1}_\la(\nu)$.
}
\end{center}
\end{figure}
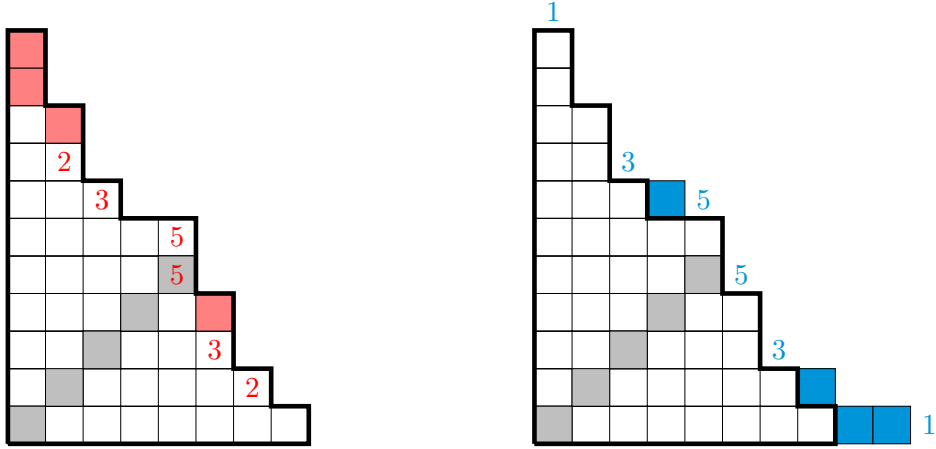
Let $\la \prec \nu$ be two partitions with Frobenius notations $\la=(a_1,\ldots,a_l|b_1,\ldots, b_l)$ and $\nu = (c_1,\ldots,c_m|c_1+1,\ldots,c_m+1)$, i.e., $\nu \in \bb{P}^{1}_{\asym}$. First observe that $l=m$ holds, otherwise we would need to add at least two cells to $\la$ in column $l+1$ which contradicts $\la \prec \nu$. The condition $\la \prec \nu$ is equivalent to the two families of inequalities
\[
c_i \geq a_i \geq c_{i+1} +1 \qquad c_i+1 \geq b_i \geq c_i,,
\]
for all $1 \leq i \leq l$ where we set $c_{l+1}=-1$. Hence the existence of the partition $\nu$ implies 
\begin{equation}
\label{eq:interlacing condition 1 asym}
b_1\geq a_1 \geq b_2 \geq \cdots \geq b_l \geq a_l \qquad \Leftrightarrow \qquad (a_1,\ldots,a_l) \prec (b_1,\ldots,b_l).
\end{equation}
Since both $(a_1,\ldots,a_l)$ and $(b_1,\ldots,b_l)$ are strict partitions, \eqref{eq:interlacing condition 1 asym} is actually equivalent to the existence of a partition $\nu$ as above.
Define $\Set^{1}_\la:=\{i: a_{i-1} > b_i > a_i \}$, where we set $a_0=\infty$. Assume that \eqref{eq:interlacing condition 1 asym} holds, then each $\nu \in \bb{P}^{1}_{\asym}$ with $\la \prec \nu$ corresponds to a subset of $\Set_\la^{1}$ as follows. First, note that $b_i=a_i$ implies $c_i=b_i$, and $b_i=a_{i-1}$ implies $c_i=b_i-1$ for all $i$, compare with Figure~\ref{fig:ex for 1 asym}. On the other hand for $i\in \Set_\la^{1}$ we can have either $c_i=b_i-1$ or $c_i=b_i$. We define $\Set^{1}_\la(\nu)$ as the set of $i\in \Set_\la^{1}$ such that $c_i=b_i$. It is immediate, that $\nu \mapsto \Set^{1}_\la(\nu)$ is a bijection between all $\nu$ of the above form and all subsets of $\Set_\la^{1}$.

Let $\mu \prec^\prime \la$ where $\mu$ has Frobenius notation $(d_1,\ldots,d_m|d_1+1,\ldots,d_m+1)$. Observe that this implies $m \in \{l,l-1\}$. In the case $m=l-1$ it is convenient to define $d_l=-1$. The condition $\mu \prec^\prime \la$  is equivalent to the family of inequalities
\[
d_i+1 \geq a_i \geq d_i, \qquad d_{i-1} \geq b_i \geq d_{i} +1,
\]
where we set $d_0=\infty$. As before, the existence of $\mu$ is therefore equivalent to \eqref{eq:interlacing condition 1 asym}.  Denote by $\Ret_\la^{1} :=\{i: b_{i+1} < a_i <b_i\}$, where we set $b_{l+1}=-1$. Each $\mu \in \bb{P}^{1}_{\asym}$ with $\mu \prec^\prime \la$ corresponds exactly to one subset of $\Ret_\la^{1}$ as shown next.  First, $a_i=b_i$ implies $d_i=a_i-1$, and $a_i=b_{i+1}$ implies $d_i=a_i$. On the other hand for $i\in \Ret_\la^{1}$ we have the two choices $d_i=a_i$ or $d_i=a_i-1$. Denote by $\Ret_\la^{1}(\mu)$ the set of all $i\in \Ret_\la^{1}$ such that $d_i=a_i-1$. It is easy to see that the map $\mu \mapsto \Ret^{1}_\la(\mu)$ is a bijection between all $\mu$ of the above form and all subsets of $\Ret_\la^{1}$.\\

It is not difficult to see, for example by comparing with Figure~\ref{fig:ex for 1 asym}, that $|\nu / \la|= |\la/\mu|$ if $|\Set^{1}_\la(\nu)| = |\Ret^{1}_\la(\mu)|$. Hence it suffices to show that $|\Set^{1}_\la|= |\Ret^{1}_\la|$ in order to prove \eqref{eq:projection 1asym partitions}.
We place the sequence $b_1,a_1,b_2,a_2,\ldots$ in a zigzag pattern and connect two integers if they are equal. For $(a_i)_i=(7,5,3,2,0)$ and $(b_i)_i=(10,7,4,2,1)$ this is shown next.
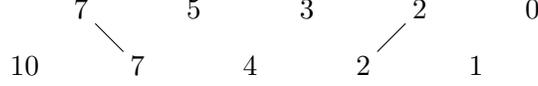
\begin{center}
\begin{tikzpicture}[scale=.75]
\node at (0,0) {$10$};
\node at (2,0) {$7$};
\node at (4,0) {$4$};
\node at (6,0) {$2$};
\node at (8,0) {$1$};
\node at (1,1) {$7$};
\node at (3,1) {$5$};
\node at (5,1) {$3$};
\node at (7,1) {$2$};
\node at (9,1) {$0$};
\draw (6.25,.25) -- (6.75,.75);
\draw (1.25,.75) -- (1.75,.25);
\end{tikzpicture}
\end{center}
Since $\Ret_\la^{1}$ corresponds to the numbers in the top row which are not connected, and $\Set_\la^{1}$ corresponds to the numbers in the bottom row which are not connected, it is immediate that $|\Set_\la^{1}|=|\Ret_\la^{1}|$. \bigskip

We define a natural bijection between the sets $\mc{D}_{\bb{P}_{\asym}^{1}}(\la,k)$ and $\mc{U}_{\bb{P}_{\asym}^{1}}(\la,k)$ denoted by $\F_{\bb{P}_{\asym}^{1},\la,k}^{\row*}$.
 For $\Ret^{1}_\la=\{r_1,\ldots,r_n\}$ and $\Set^{1}_\la=\{s_1,\ldots,s_{n}\}$  we define $\F_{\bb{P}_{\asym}^{1},\la,2k}^{\row*}$ to map the partition $\mu$  with $\Ret_\la^{1}(\mu)= \{r_{i_1},\ldots,r_{i_k}\}$ to the partition $\nu$ with $\Set_\la^{1}(\nu)=\{s_{i_1},\ldots,s_{i_k}\}$.

\subsubsection{Proof of \eqref{eq:projection -1asym partitions}}
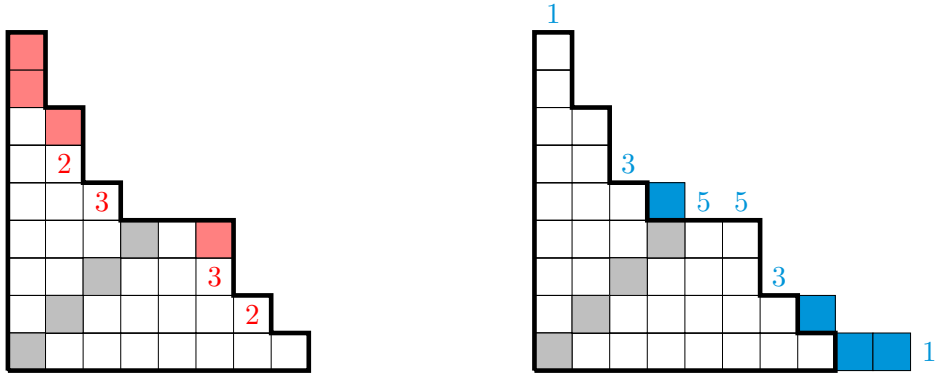
\begin{figure}[h]
\begin{center}
\begin{tikzpicture}

\tyng(0,0,8,7,6,6,3,2,2,1,1)
{\Yred
\tgyoung(0cm,0cm,,,,:::::;,,,:;,;,;)
}
{\Ygray
\tgyoung(0cm,0cm,;,:;,::;,:::;)
}
\node at (3.25,.75) {\color{red}{$2$}};
\node at (2.75,1.25) {\color{red}{$3$}};
\node at (1.25,2.25) {\color{red}{$3$}};
\node at (.75,2.75) {\color{red}{$2$}};
\draw[line width=1.75pt] (0,0) -- (4,0) -- (4,.5) -- (3.5,.5) -- (3.5,1) -- (3,1) -- (3,2) -- (1.5,2) -- (1.5,2.5) -- (1,2.5) -- (1,3.5) -- (.5,3.5) -- (.5,4.5) -- (0,4.5) -- (0,0);

\begin{scope}[xshift=7cm]
\tyng(0,0,8,7,6,6,3,2,2,1,1)
{\Ycyan
\tgyoung(0cm,0cm,::::::::;;,:::::::;,,,:::;)
}
{\Ygray
\tgyoung(0cm,0cm,;,:;,::;,:::;)
}
\node at (5.25,.25) {\color{myblue}{$1$}};
\node at (3.25,1.25) {\color{myblue}{$3$}};
\node at (2.75,2.25) {\color{myblue}{$5$}};
\node at (2.25,2.25) {\color{myblue}{$5$}};
\node at (1.25,2.75) {\color{myblue}{$3$}};
\node at (.25,4.75) {\color{myblue}{$1$}};
\draw[line width=1.75pt] (0,0) -- (4,0) -- (4,.5) -- (3.5,.5) -- (3.5,1) -- (3,1) -- (3,2) -- (1.5,2) -- (1.5,2.5) -- (1,2.5) -- (1,3.5) -- (.5,3.5) -- (.5,4.5) -- (0,4.5) -- (0,0);
\end{scope}
\end{tikzpicture}
\caption{\label{fig:ex for -1 asym}
 On the left a sketch of all $\mu \in \bb{P}^{-1}_{\asym}$ with $\mu \prec^\prime \la=(7,5,3,2|8,5,2,0)$, where the boxes along the diagonal are marked gray, the boxes which have to be removed are in red and a box labelled $i$ is removed if $i \in \Ret^{-1}_\la(\mu)$.
On the right a sketch of all $\nu \in \bb{P}^{-1}_{\asym}$ with $\nu\succ \la$, where boxes are in blue if they have to be added and a label $i$ marks a box which is added if $i\in \Set^{-1}_\la(\nu)$.}

\end{center}
\end{figure}
Let $\la=(a_1,\ldots,a_l|b_1,\ldots,b_l)$ and $\nu=(c_1+1,\ldots,c_m+1|c_1,\ldots,c_m) \in \bb{P}_\asym^{-1}$ such that $\la \prec \nu$. First observe that we have $m \in \{l,l+1\}$ and for convenience define $c_{l+1}=-\infty$ if $m=l$.
By definition, the condition $\la \prec \nu$ is equivalent to the two families of inequalities
\[
c_{i}+1 \geq a_i \geq c_{i+1}+2, \qquad c_i \geq b_i \geq c_i-1,
\]
for all $1 \leq i \leq l$. Hence the existence of a partition $\nu$ as above is equivalent to
\begin{equation}
\label{eq:interlacing condition -1 asym}
b_1+2\geq a_1 \geq b_2+2 \geq \cdots \geq b_l+2 \geq a_l \qquad \Leftrightarrow \qquad (a_1,\ldots,a_l) \prec (b_1+2,\ldots,b_l+2),
\end{equation} 
where the equivalence holds since $(a_1,\ldots,a_l)$ and $(b_1,\ldots,b_l)$ are strict partitions.
Define $\Set_\la^{-1} = \{i \leq l+1: a_{i-1} > b_i+2 > a_{i}\}$, where we set $a_0=\infty$, $b_{l+1}=-1$ and $a_{l+1}=-\infty$. We show that each $\nu$ with $\la \prec \nu$ corresponds to a subset of $\Set_\la^{-1}$ as follows. First note that $b_i+2=a_i$ implies $c_i=b_i+1$, and $b_i+2=a_{i-1}$ implies $c_i=b_i$. On the other hand if $a_{i-1} > b_i+2 > a_{i}$ holds, we have two allowed choices for $c_i$, namely $b_i$ and $b_{i}+1$. In particular we can have $c_{l+1}=1$ if and only if $a_l\geq 2$ which is equivalent to $(l+1) \in \Set_\la^{-1}$. We define $\Set_\la^{-1}(\nu)$ as the set of $i \in \Set_\la^{-1}$ such that $c_i=b_i+1$. By the above argument it is immediate that $\nu \mapsto \Set_\la^{-1}(\nu)$ is a bijection between all $\nu$ as above and all subsets of $\Set_\la^{-1}$.

Let $\mu=(d_1+1,\ldots,d_m+1|d_1,\ldots,d_m)$  be a partition with $\mu \prec^\prime \la$. The existence of such a $\mu$ is equivalent to the family of inequalities
\[
d_{i}+2 \geq a_i \geq d_{i}+1, \qquad d_{i-1}-1 \geq b_i \geq d_i,
\]
which is again equivalent to \eqref{eq:interlacing condition -1 asym}. Denote by  $\Ret_\la^{-1} := \{ i\leq l: b_i+2 > a_i > b_{i+1}+2\}$. Each $\mu \in \bb{P}_{\asym}^{-1}$ with $\mu \prec^\prime \la$ corresponds to a subset of $\Ret_\la^{-1}$ as follows. First, the equality $a_i=b_i+2$ implies $d_i=a_i-2$, and $a_i = b_{i+1}+2$ implies $d_i=a_i-1$. Note that the case $a_i < b_{i+1}+2$ can only appear for $i=l$ and $a_l=0$ which implies $m=l-1$. For all $i \in \Ret_\la^{-1}$ we have however two possible choices for $d_i$, namely $a_i-1$ or $a_i-2$. Denote by $\Ret_\la^{-1}(\mu)$ the set of all $i \in \Ret_\la^{-1}$ such that $d_i=a_i-2$. Then the above shows that the map $\mu \mapsto \Ret_\la^{-1}(\mu)$ is a bijection between all such partitions $\mu$ and all subsets of $\Ret_\la^{-1}$.\bigskip

Denote by $\nu_0$ the partition such that $\Set_\la^{-1}(\nu_0)=\emptyset$. Then we have
\[
|\nu / \la| = |\nu_0/\la| +2 |\Set_\la^{-1}(\nu)|, \qquad
|\la / \mu| = |\nu_0/\la| +2 |\Ret_\la^{-1}(\mu)|.
\]
Hence in order to prove \eqref{eq:projection -1asym partitions} it suffices to show $|\Set_\la^{-1}| = |\Ret_\la^{-1}|+1$.
Similar to the previous subsection we place the sequence $b_1+2,a_1,b_2+2,a_2,\ldots,b_l+2,a_l,b_{l+1}+2$ in a zigzag pattern and connect two integers if they are equal and connect $a_l$ and $b_{l+1}+2$ if $a_l\leq b_{l+1}+2=1$. For $(a_i)_i=(7,5,3,2)$ and $(b_i)_i=(8,5,2,0,-1)$ this is shown next.
\begin{center}
\begin{tikzpicture}[scale=.75]
\node at (0,0) {$10$};
\node at (2,0) {$7$};
\node at (4,0) {$4$};
\node at (6,0) {$2$};
\node at (8,0) {$1$};
\node at (1,1) {$7$};
\node at (3,1) {$5$};
\node at (5,1) {$3$};
\node at (7,1) {$2$};
\draw (6.25,.25) -- (6.75,.75);
\draw (1.25,.75) -- (1.75,.25);
\end{tikzpicture}
\end{center}
Again the unconnected numbers in the top row correspond to the elements in $\Ret_\la^{-1}$ and the unconnected numbers in the bottom row correspond to the elements of $\Set_\la^{-1}$, which proves the claim.\bigskip

We define two natural bijections between the sets  $\mc{D}_{\bb{P}_{\asym}^{-1}}(\la,2k) \cup \mc{D}_{\bb{P}_{\asym}^{-1}}(\la,2k-2)$ and $\mc{U}_{\bb{P}_{\asym}^{-1}}(\la,2k)$ denoted by $\F_{\bb{P}_{\asym}^{-1},\la,2k}^{\row*}$ and $\F_{\bb{P}_{\asym}^{-1},\la,2k}^{\col*}$.  For $\Ret^{-1}_\la=\{r_1,\ldots,r_n\}$ and $\Set^{-1}_\la=\{s_0,\ldots,s_{n}\}$ these are defined as
\[
\F_{\bb{P}_{\asym}^{-1},\la,2k}^{\row*}:
\{r_{i_1},\ldots,r_{i_j}\} \mapsto
\begin{cases}
\{s_{i_1},\ldots,s_{i_j}\} \quad &\text{if }j=k,\\
\{s_{0},s_{i_1},\ldots,s_{i_j}\} &\text{if }j=k-1,
\end{cases}
\]
and
\[
\F_{\bb{P}_{\asym}^{-1},\la,2k}^{\col*}:
\{r_{i_1},\ldots,r_{i_j}\} \mapsto
\begin{cases}
\{s_{i_1-1},\ldots,s_{i_j-1}\} \quad &\text{if }j=k,\\
\{s_{i_1-1},\ldots,s_{i_j-1},s_n\} &\text{if }j=k-1,
\end{cases}
\]
where we identified the partitions $\mu,\nu$ with the corresponding sets $\Ret_\la^{-1}(\mu)$ and $\Set_\la^{-1}(\nu)$.

\subsection{Triangular growths}
\label{sec:triangular}

In this Section we adapt Fomin's growth diagrams to our framework. Instead of symmetric matrices we use ``upper triangular'' arrays.

\begin{defi}
\label{def:triangular growth}
Let $C=(c_{i,j})_{1 \leq i\leq j\leq n}$ be a triangular array with non-negative integer entries. A \deff{triangular growth associated with $C$} is an assignment  $\Lambda=(\Lambda_{i,j})$ of partitions to the vertices $(i,j)\in [0,n]\times[0,m]$ with $i \leq j$ such that 
\begin{itemize}
\item for $i<n$ holds $\Lambda_{i,j} \prec \Lambda_{i+1,j}$, i.e., $\Lambda_{i+1,j}/\Lambda_{i,j}$ is a horizontal strip,
\item for $j<m$ holds $\Lambda_{i,j} \prec \Lambda_{i,j+1}$, i.e., $\Lambda_{i,j+1}/\Lambda_{i,j}$ is a horizontal strip,
\item the size $|\Lambda_{i,j}|$ is equal to
\begin{equation}
\label{eq:triangular growth sizes}
|\Lambda_{i,j}| = \sum_{1 \leq k < l \leq i} c_{k,l} + \sum_{\substack{1 \leq k \leq i \\ k \leq l \leq j}}c_{k,l}.
\end{equation}
\end{itemize}
We say that $\Lambda$ is a \deff{triangular dual growth}, if $\Lambda_{i+1,j}/\Lambda_{i,j}$ is a vertical strip instead of a horizontal strip. Since $C$ is determined by $\Lambda$ we further omit referring to it.
\end{defi}

\begin{ex}
\label{ex:triangular growths}
The triangular array $C= \begin{array}{ccc}
0 & 0 & 1 \\ & 1 & 0 \\ && 0
\end{array}$  has four triangular growths associated to it as shown next.
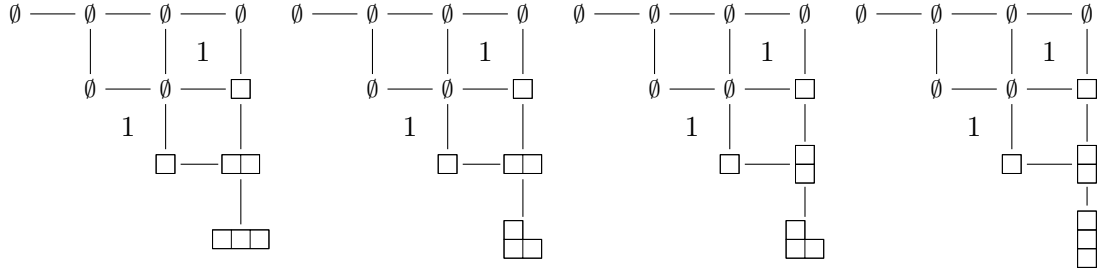
\begin{center}
\begin{tikzpicture}[scale=1]
	\node at (2.5,-0.5) {$1$};
	\node at (1.5,-1.5) {$1$};
	{\footnotesize
		\node at (0,0) {$\emptyset$};
		\node at (1,0) {$\emptyset$};
		\node at (2,0) {$\emptyset$};
		\node at (3,0) {$\emptyset$};
		\node at (1,-1) {$\emptyset$};
		\node at (2,-1) {$\emptyset$};
	\Yboxdim{0.25cm}
		\node at (3,-1) {\yng(1)};
		\node at (2,-2) {\yng(1)};
		\node at (3,-2) {\yng(2)};
		\node at (3,-3) {\yng(3)};
	}
	\draw (.2,0) -- (.8,0);
	\draw (1.2,0) -- (1.8,0);
	\draw (1.2,-1) -- (1.8,-1);
	\draw (2.2,0) -- (2.8,0);
	\draw (2.2,-1) -- (2.8,-1);
	\draw (2.2,-2) -- (2.7,-2);

	\draw (1,-.2) -- (1,-.8);
	\draw (2,-.2) -- (2,-.8);
	\draw (2,-1.2) -- (2,-1.8);
	\draw (3,-.2) -- (3,-.8);
	\draw (3,-1.2) -- (3,-1.8);
	\draw (3,-2.2) -- (3,-2.8);
	
		\begin{scope}[xshift=3.75cm]
	\node at (2.5,-0.5) {$1$};
	\node at (1.5,-1.5) {$1$};
	{\footnotesize
		\node at (0,0) {$\emptyset$};
		\node at (1,0) {$\emptyset$};
		\node at (2,0) {$\emptyset$};
		\node at (3,0) {$\emptyset$};
		\node at (1,-1) {$\emptyset$};
		\node at (2,-1) {$\emptyset$};
	\Yboxdim{0.25cm}
		\node at (3,-1) {\yng(1)};
		\node at (2,-2) {\yng(1)};
		\node at (3,-2) {\yng(2)};
		\node at (3,-3) {\yng(2,1)};
	}
	\draw (.2,0) -- (.8,0);
	\draw (1.2,0) -- (1.8,0);
	\draw (1.2,-1) -- (1.8,-1);
	\draw (2.2,0) -- (2.8,0);
	\draw (2.2,-1) -- (2.8,-1);
	\draw (2.2,-2) -- (2.7,-2);

	\draw (1,-.2) -- (1,-.8);
	\draw (2,-.2) -- (2,-.8);
	\draw (2,-1.2) -- (2,-1.8);
	\draw (3,-.2) -- (3,-.8);
	\draw (3,-1.2) -- (3,-1.8);
	\draw (3,-2.2) -- (3,-2.7);
		\end{scope}
		
		\begin{scope}[xshift=7.5cm]
	\node at (2.5,-0.5) {$1$};
	\node at (1.5,-1.5) {$1$};
	{\footnotesize
		\node at (0,0) {$\emptyset$};
		\node at (1,0) {$\emptyset$};
		\node at (2,0) {$\emptyset$};
		\node at (3,0) {$\emptyset$};
		\node at (1,-1) {$\emptyset$};
		\node at (2,-1) {$\emptyset$};
	\Yboxdim{0.25cm}
		\node at (3,-1) {\yng(1)};
		\node at (2,-2) {\yng(1)};
		\node at (3,-2) {\yng(1,1)};
		\node at (3,-3) {\yng(2,1)};
	}
	\draw (.2,0) -- (.8,0);
	\draw (1.2,0) -- (1.8,0);
	\draw (1.2,-1) -- (1.8,-1);
	\draw (2.2,0) -- (2.8,0);
	\draw (2.2,-1) -- (2.8,-1);
	\draw (2.2,-2) -- (2.8,-2);

	\draw (1,-.2) -- (1,-.8);
	\draw (2,-.2) -- (2,-.8);
	\draw (2,-1.2) -- (2,-1.8);
	\draw (3,-.2) -- (3,-.8);
	\draw (3,-1.2) -- (3,-1.7);
	\draw (3,-2.3) -- (3,-2.7);
		\end{scope}

		\begin{scope}[xshift=11.25cm]
	\node at (2.5,-0.5) {$1$};
	\node at (1.5,-1.5) {$1$};
	{\footnotesize
		\node at (0,0) {$\emptyset$};
		\node at (1,0) {$\emptyset$};
		\node at (2,0) {$\emptyset$};
		\node at (3,0) {$\emptyset$};
		\node at (1,-1) {$\emptyset$};
		\node at (2,-1) {$\emptyset$};
	\Yboxdim{0.25cm}
		\node at (3,-1) {\yng(1)};
		\node at (2,-2) {\yng(1)};
		\node at (3,-2) {\yng(1,1)};
		\node at (3,-3) {\yng(1,1,1)};
	}
	\draw (.2,0) -- (.8,0);
	\draw (1.2,0) -- (1.8,0);
	\draw (1.2,-1) -- (1.8,-1);
	\draw (2.2,0) -- (2.8,0);
	\draw (2.2,-1) -- (2.8,-1);
	\draw (2.2,-2) -- (2.8,-2);

	\draw (1,-.2) -- (1,-.8);
	\draw (2,-.2) -- (2,-.8);
	\draw (2,-1.2) -- (2,-1.8);
	\draw (3,-.2) -- (3,-.8);
	\draw (3,-1.2) -- (3,-1.7);
	\draw (3,-2.3) -- (3,-2.55);
		\end{scope}
\end{tikzpicture}
\end{center}
\end{ex}

Analogously to usual growths, we associate to a triangular (dual) growth $\Lambda$ the semistandard Young tableau $P(\Lambda)$ by reading of the right-most column, i.e., the shape of $P(\Lambda)$ restricted to entries at most $i$ is $\Lambda_{i,n}$. The according tableau for the second triangular growth in the above example is
\[
P(\Lambda)=\young(12,3) \;.
\]
For a triangular (dual) growth $\Lambda$ associated to $C$ with $P=P(\Lambda)$ we write
\[
\Lambda: C \rightarrow P
\]
and say that ``$\Lambda$ is a triangular (dual) growth from $C$ to $P$''. Each full square of a triangular (dual) growth is of the form \eqref{eq:growth square} or \eqref{eq:dual growth square} respectively. The partial squares on the boundary are of the form as shown left (resp., right) 
\begin{equation}
\label{eq:partial growth square}
\begin{tikzpicture}[baseline=-3.5ex]
\draw (1.2,-.1) --node[right]{\rotatebox{-90}{$\prec$}} (1.2,-.9);
\draw (0.1,0.2) --node[above]{$\prec$} (.9,0.2);
\node at (-0.2,0.2) {$\mu$};
\node at (1.2,0.2) {$\la$};
\node at (1.2,-1.2) {$\nu$};
\node at (.5,-.5) {$c$};
\end{tikzpicture}
\quad \text{with } \mu \in \mc{D}_{\bb{P}}(\la,k-a),
\qquad\qquad
\begin{tikzpicture}[baseline=-3.5ex]
\draw (1.2,-.1) --node[right]{\rotatebox{-90}{$\prec$}} (1.2,-.9);
\draw (0.1,0.2) --node[above]{$\prec^\prime$} (.9,0.2);
\node at (-0.2,0.2) {$\mu$};
\node at (1.2,0.2) {$\la$};
\node at (1.2,-1.2) {$\nu$};
\node at (.5,-.5) {$c$};
\end{tikzpicture}
\quad \text{with } \mu \in \mc{D}_{\bb{P}}^*(\la,k-a),
\end{equation}
where $c \in \bb{N}$ and $\nu \in \mc{U}_{\bb{P}}(\la,k)$.
By extending the local (dual) growth rules as shown next, we can always choose a unique triangular growth from $C$ to $P$.

\begin{defi}
Let $X$ be a set of partitions and $I_X(\la,k) \subseteq [0,k]$ for all $\la \in \bb{P}$ and $k \in \bb{N}$. A set of \deff{local triangular growth rules} $\F_\bullet$ is a set of local growth rules $\F_\bullet$ together with a family of bijections
\[
\F_{X,\la,k}: \bigcup_{i \in I_X(\la,k)} \mc{D}_{X}(\la,i) \rightarrow  \mc{U}_{X}(\la,k),
\]
for each $\la \in \bb{P}$ and $k \in \bb{N}$. Analogously a set of \deff{local triangular dual growth rules} $\F_\bullet^*$ is a set of local dual growth rules  $\F_\bullet^*$  together with a family of bijections
\[
\F_{X,\la,k}^*: \bigcup_{i \in I_X(\la,k)} \mc{D}_{X}^*(\la,i) \rightarrow  \mc{U}_{X}(\la,k),
\]
for each $\la \in \bb{P}$ and $k \in \bb{N}$.
\end{defi}

We have already seen local triangular growth rules implicitly in the previous section. For example Section~\ref{sec:proof all partitions} considers the case $X=\bb{P}$ and $I_{\bb{P}}(\la,k)=[0,k]$ and Section~\ref{sec:proof even partitions} the case $X=\bb{P}_e$ and $I_{\bb{P}_e}(\la,2k+\odd(\la))=\{\odd(\la),\odd(\la)+2,\ldots,\odd(\la)+2k\}$ and $I_{\bb{P}_e}(\la,2k+1+\odd(\la))=\emptyset$.\medskip

Given a set of local triangular (dual) growth rules $\F_\bullet$, we call a triangular (dual) growth $\Lambda$ a \deff{triangular (dual) $\F_\bullet$-growth diagram} if each full square satisfies $\nu = \F_{\la,\rho,|(\la\cap\rho)/\mu|+a}(\mu)$ (resp., $\nu = \F^*_{\la,\rho,|(\la\cap\rho)/\mu|+a}(\mu)$), where $\mu,\la,\rho,\nu,a$ are defined as in \eqref{eq:growth square} (resp., \eqref{eq:dual growth square}) and each partial square satisfies $\nu=\F_{X,\la,|\la/\mu|+c}$ (resp., $\nu=\F^*_{X,\la,|\la/\mu|+c}$) where $\la,\mu,\nu,c$ are defined as on the left (resp., right) of \eqref{eq:partial growth square}. As in Section~\ref{sec:growths} the important observation is that each set of local triangular (dual) growth rules $\F_\bullet$ yields a bijection between the set of arrays $C=(c_{i,j})_{1 \leq i\leq j \leq n}$ where the non-diagonal entries are either in $\bb{N}$ or $\{0,1\}$ and the diagonal entries are either $0$, in $2\bb{N}$ or in $\bb{N}$ respectively\footnote{The allowed values for the entries of $C=(c_{i,j})$ depend on the explicit situation, i.e., which of the Littlewood identities we want to prove.}, and the set of SSYTs with shapes in $X$ and entries at most $n$. Given an array $C$ we construct a triangular (dual) $\F_\bullet$-growth diagram as before by using the local triangular (dual) growth rules.\\

Similarly to Section \ref{sec:growths} we can also translate the local triangular (dual) growth rules $\F_\bullet$ into an $\F_\bullet$-insertion algorithm. Let $C$ be a triangular array as before, denote by $T$ the tableau obtained after inserting the first $i-1$ columns and $\wh{T}$ the tableau obtained by inserting the $i$-th column of $C$ into $T$. Then $\wh{T}$ is obtained by the following two steps.
%
%
%
%
%
\begin{enumerate}
\item We $\F_\bullet$-insert the first $i-1$ entries of the $i$-th column as defined in Lemma~\ref{lem:insertion algorithm} or \cite[Lemma 2.10]{FriedenSchreierAigner24plus} to obtain $\wh{T}$ restricted to entries at most $i-1$.
\item We insert an entry $i$ in all cells of $\F_{X,\wh{T}^{(i-1)},k}(T^{(i-1)})/T^{(i-1)}$, where 
\[
k=\left|T^{(i-1)}/\wh{T}^{(i-1)}\right|+A_{i,i}=\sum_{j=1}^iA_{j,i}.
\]
\end{enumerate}

\begin{rem}
We can use the above translation to insertion algorithms to relate our growth diagram proofs to well known bijections.
\begin{enumerate}
\item The Littlewood identity \eqref{eq:littlewood even col} is proved by Burges ``INSERT1'' algorithm \cite[Section 2]{Burge74}. Let $(u_k,v_k),(u_{k+1},v_{k+1}),\ldots,(u_l,v_l)$ be all pairs of the two-line array with $i=u_k=\cdots=u_l$. It follows by the proof of the theorem of Section 2 of \cite[p. 20]{Burge74} that none of the $u_j$ will be bumped when inserting these pairs by ``INSERT1''. This implies that we can first row insert the $v_j$ and then add all $u_j$, i.e. the entry $i$, at once. Hence this algorithm is equivalent to using triangular growths where we use row insertion for all local growth rules.

\item  The Littlewood identity \eqref{eq:littlewood identity} is usually obtained from the symmetry of RSK, see \eqref{eq:sym of RSK}. 
 Since the bijection $\F_{\bb{P},\la,k}$ is defined by $\F_{\bb{P},\la,k} = \F_{\la,\la,k}$, it follows that the triangular $\F_\bullet$-growth diagrams are equal to the upper half of a regular growth diagram of a symmetric matrix. In particular this covers the classical proof via row or column insertion.

\item Burge argues in \cite[Section 3]{Burge74} that the Littlewood identity \eqref{eq:littlewood even row} can be obtained by applying column insertion, which he calls ``INSERT*'', to a symmetric matrix with even diagonal entries, or equivalently by a simplified insertion algorithm called ``INSERT2''. Since the definition of $\F_{\bb{P}_e,\la,k}$ is based on the local growth rules $\F_\bullet$ we obtain the same correspondence by the above argument by using triangular growths where we choose column insertion $\F_\bullet^\col$. 

\item The Littlewood identity  \eqref{eq:littlewood 1 asym} is proved by Burges ``INSERT3'' algorithm. Let $(u_k,v_k)$, $(u_{k+1},v_{k+1})$,$\ldots,$ $(u_l,v_l)$ be all pairs of the two-line array with $i=u_k=\cdots=u_l$. Since $i>v_j$ for all $k\leq j \leq l$ no $i$ entry is bumped while inserting these pairs via ``INSERT3''. Hence we can equivalently first insert $v_k,v_{k+1},\ldots,v_l$ as defined by ``INSERT3'' and then insert all $i$ entries. By definition, inserting $v_k,v_{k+1},\ldots,v_l$ via ``INSERT3'' corresponds to the dual row insertion $\F_{\bullet}^{\row*}$ since dual row insertion is reverse traceable. It is immediate by the definition of ``INSERT3'' that inserting  all $i$ entries is equivalent to the definition of $\F_{\bb{P}_\asym^{1},\la,k}^{\row*}$. Hence ``INSERT3'' is equivalent to using triangular dual growths with dual row insertion $\F_{\bullet}^{\row*}$ and  $\F_{\bb{P}_\asym^{1},\la,k}^{\row*}$ respectively.

\item Burge proves the Littlewood identity \eqref{eq:littlewood -1 asym} by the ``INSERT4'' algorithm. As above, we can for each column first insert the $v_j$ entries and then the $u_j$. By definition of ``INSERT4'' the $v_j$ are inserted by dual column insertion. Furthermore it is immediate that inserting the $u_j$ entries is equal to applying  $\F_{\bb{P}_\asym^{-1},\la,k}^{\row*}$. Hence Burge's ``INSERT4'' algorithm is equivalent to using triangular dual growth diagrams together with dual column insertion $\F_{\bullet}^{\col*}$ and $\F_{\bb{P}_\asym^{-1},\la,k}^{\row*}$.
\end{enumerate}
\end{rem}

Similar to Section~\ref{sec:skew} we can extend  the framework of triangular (dual) growth diagrams to obtain bijective proofs for the following skew Littlewood identities, compare also to \cite[Ch. I, Section 5, Ex. 27]{Macdonald95}.

\begin{thm}
\label{thm:skew Littlewood identities}
Let $\x=(x_1,\ldots,x_n)$ be a sequence of variables and let $\la$ be a partition. Then,
\begin{align}
\label{eq:skew littlewood even col}
\sum_{\nu \in \bb{P}_e^\prime} s_{\nu/\la}(\x) &= \prod_{1 \leq i<j \leq n}\frac{1}{1-x_ix_j}\sum_{\mu \in \bb{P}_e^\prime} s_{\la/\mu}(\x),\\
\label{eq:skew littlewood identity}
\sum_{\nu \in \bb{P}} s_{\nu/\la}(\x) &= \prod_{i=1}^n\frac{1}{1-x_i}\prod_{1\leq i<j \leq n}\frac{1}{1-x_ix_j} \sum_{\mu \in \bb{P}} s_{\la/\mu}(\x),\\
\label{eq:skew littlewood even row}
\sum_{\nu \in \bb{P}_e} s_{\nu/\la}(\x) &= \prod_{i=1}^n\frac{1}{1-x_i^2}\prod_{1\leq i<j \leq n}\frac{1}{1-x_ix_j}\sum_{\mu \in \bb{P}_e} s_{\la/\mu}(\x),\end{align}
\begin{align}
\label{eq:skew littlewood 1 asym}
\sum_{\nu \in \bb{P}_\asym^{1}} s_{\nu/\la}(\x) &=\prod_{1 \leq i<j \leq n}(1+x_ix_j)\sum_{\mu \in \bb{P}_\asym^{-1}} s_{\la^\prime/\mu}(\x), \\
\label{eq:skew littlewood -1 asym}
\sum_{\nu \in \bb{P}_\asym^{-1}} s_{\nu/\la}(\x) &=\prod_{i=1}^n(1+x_i^2)\prod_{1 \leq i<j \leq n}(1+x_ix_j)\sum_{\mu \in \bb{P}_\asym^{1}} s_{\la^\prime/\mu}(\x). 
\end{align}
\end{thm}

More explicitely, for a triangular array $C=(c_{i,j})_{1 \leq i\leq j\leq n}$ as before and a tableau $S \in \SSYT_{\la/\mu}(n)$ (resp. $S \in \SSYT_{\la/\mu^\prime}^*(n)$) we construct a triangular (dual) growth diagram as before with the only difference that we set $\Lambda_{0,i} = S^{(i)}$ for $0 \leq i \leq n$ instead of $\Lambda_{0,i} = \emptyset$. The following example illustrates this for the case of \eqref{eq:skew littlewood -1 asym}.

\begin{ex}
The triangular dual growth diagram associated to the pair $(C,S)$ for
\[
C = \begin{array}{cccc}
0 & 1 & 0 & 0 \\
  & 0 & 0 & 1 \\
  &   & 2 & 0 \\
  &   &   & 0
\end{array}, \qquad
S= \begin{tikzpicture}[baseline=.4cm]
\tyoung(0cm,0cm,:::34,:123,2)
\Ygray
\tgyoung(0cm,0cm,;;;,;)
\end{tikzpicture}\,, \qquad
\]
where we use the local triangular dual growth rules $\F_\bullet^{\row*}$ and $\F_{\bb{P}_\asym^{-1},\la,k}^{\row*}$ is shown next.
\begin{center}
\begin{tikzpicture}[scale=2.25]
\Yboxdim{.25cm}
\node at (1.5,-.5) {$1$};
\node at (3.5,-1.5) {$1$};
\node at (2.5,-2.5) {$2$};

\node at (0,0) {\yng(3,1)};
\node at (1,0) {\yng(3,2)};
\node at (2,0) {\yng(3,3,1)};
\node at (3,0) {\yng(4,4,1)};
\node at (4,0) {\yng(5,4,1)};

\node at (1,-1) {\yng(3,3)};
\node at (2,-1) {\yng(4,3,1,1)};
\node at (3,-1) {\yng(4,4,2,1)};
\node at (4,-1) {\yng(5,4,2,1)};

\node at (2,-2) {\yng(5,4,2,1)};
\node at (3,-2) {\yng(5,4,3,2)};
\node at (4,-2) {\yng(5,5,3,2)};

\node at (3,-3) {\yng(6,5,4,2,1)};
\node at (4,-3) {\yng(6,5,5,2,1)};

\node at (4,-4) {\yng(6,5,5,3,1)};

\draw (.25,0) -- (.75,0);
\draw (1.25,0) -- (1.75,0);
\draw (2.25,0) -- (2.7,0);
\draw (3.3,0) -- (3.65,0);

\draw (1.3,-1) -- (1.75,-1);
\draw (2.25,-1) -- (2.7,-1);
\draw (3.3,-1) -- (3.65,-1);

\draw (2.3,-2) -- (2.65,-2);
\draw (3.3,-2) -- (3.65,-2);

\draw (3.35,-3) -- (3.6,-3);

\draw (1,-0.2) -- (1,-0.8);
\draw (2,-0.25) -- (2,-0.7);
\draw (3,-0.25) -- (3,-0.7);
\draw (4,-0.25) -- (4,-0.7);

\draw (2,-1.3) -- (2,-1.7);
\draw (3,-1.3) -- (3,-1.7);
\draw (4,-1.3) -- (4,-1.7);

\draw (3,-2.3) -- (3,-2.65);
\draw (4,-2.3) -- (4,-2.65);

\draw (4,-3.35) -- (4,-3.65);

\end{tikzpicture}
\end{center}
By reading the last column of the above triangular dual growth diagram we obtain
\[
P = \begin{tikzpicture}[baseline=.9cm]
\tyoung(0cm,0cm,:::::3,::::2,:1233,124,3)
\Ygray
\tgyoung(0cm,0cm,;;;;;,;;;;,;)
\end{tikzpicture}\,.
\]
\end{ex}

\section{A general framework}
\label{sec:vier}
In the final section we describe the general set up for which our new approach yields an algebraic proof of Littlewood type identities. \bigskip

Let $P$ be a ranked poset 
 together with two functions $\psi,\psi^\perp:P^2 \rightarrow K$, where $K$ is a field and $\psi(\la, \mu) \neq 0$ or $\psi^\perp(\la, \mu) \neq 0$ respectively implies $\mu \subseteq \la$. Typically we write $\psi_{\la/\mu}$ and $\psi^\perp_{\la/\mu}$ instead of  $\psi(\la, \mu)$ or $\psi^\perp(\la, \mu)$.
We define the operators $U_x$ and $D_y$ called \deff{up operator} and \deff{down operator} respectively on the $K \lb x,y \rb$-module $P\lb x,y\rb$ of formal power series in $x,y$ whose coefficients are finite formal sums of $P$ over $K$ as
\[
U_x \la = \sum_{\nu} \psi_{\nu/\la} x^{\rk(\nu)-\rk(\la)}\nu, \qquad \qquad
D_y \la = \sum_{\mu } \psi^\perp_{\la/\mu} y^{\rk(\la)-\rk(\mu)}\mu,
\]
for partitions $\la$ and linearly extend it to $P\lb x,y\rb$.
For each pair $\la,\mu \in P$ and $\x=(x_1,\ldots,x_n)$ and $\y=(y_1,\ldots,y_m)$ we define the two polynomials
\[
s_{\la/\mu}(\x) := \left\langle U_{x_n} \cdots U_{x_1} \mu , \la \right\rangle, \qquad \qquad
s_{\la/\mu}^\perp(\y) := \left\langle D_{y_1} \cdots D_{y_m} \la , \nu \right\rangle,
\]
where $\left\langle \la, \mu \right\rangle:=\delta_{\la,\mu}$ for any elements $\la,\mu \in P$ and $\left\langle \cdot, \cdot \right\rangle$ is linearly extended in the first component to a $K\lb x,y \rb$-linear functional on $P\lb x,y \rb$. 
Note that the polynomials $s_{\la/\mu}(\x)$ and $s_{\la/\mu}^\perp(\y)$ are symmetric polynomials in $\x$ or $\y$ respectively if the up and down operators satisfy the commutation relation $U_{x_i}U_{x_j}=U_{x_j}U_{x_i}$ and $D_{y_i}D_{y_j}=D_{y_j}D_{y_i}$ respectively.

\begin{prop}
\label{prop: general setting}
\begin{enumerate}
\item Let $f(x,y) \in K(x,y)$ be a rational function. The up and down operator satisfy the commutation relation
\begin{equation}
\label{eq:general commutation relation}
D_yU_x = f(x,y) U_xD_y,
\end{equation}
if and only if the polynomials $s_\la(\x)$ and $s_\la^\perp(\y)$ satisfy the Cauchy type identities
\begin{equation}
\label{eq:general skew Cauchy}
\sum_{\nu} s_{\nu/\rho}(\x)s_{\nu/\la}^\perp(\y)
 = \prod_{\substack{1 \leq i \leq n \\ 1 \leq j \leq m}}f(x_i,y_j)
 \sum_{\mu} s_{\la/\mu}(\x)s_{\rho/\mu}^\perp(\y) ,
\end{equation}
for all $\la,\rho \in P$.
\item Let $g(x) \in K\lb x\rb$ be a formal power series and $\varphi:P\lb x \rb \rightarrow K\lb x\rb$ a $K\lb x\rb$-linear functional. If the up and down operator satisfy additionally to \eqref{eq:general commutation relation} the projection identity
\begin{equation}
\label{eq:general projection identity}
\varphi \circ U_x = g(x) \cdot \varphi \circ D_x
\end{equation}
then the polynomials $s_\la(\x)$ satisfy the identity
\begin{equation}
\label{eq:general Littlewood identity}
\sum_{\nu} \varphi(\nu) s_{\nu/\la}(\x) = \prod_{i=1}^n g(x_i) \prod_{1 \leq i < j \leq n} f(x_i,x_j) \sum_{\mu}\varphi(\mu)s_{\la/\mu}^\perp(\x),
\end{equation}
for all $\la \in P$.
\end{enumerate}
\end{prop}
\begin{proof}
The first part is classical and an immediate generalisation of the paragraph after Theorem~\ref{thm:skew Cauchy identities}.
For the second part, we have by definition
\begin{multline*}
\sum_{\nu} \varphi(\nu) s_{\nu/\la}(\x)
= \sum_\nu \varphi(\nu) \left\langle U_{x_n} \cdots U_{x_1} \la , \nu \right\rangle
= \varphi \circ U_{x_n} \cdots U_{x_1} \la 
 \\
 = g(x_n) \cdot \varphi \circ D_{x_n} U_{x_{n-1}} \cdots U_{x_1} \la
 \\
 = \cdots 
 = g(x_n) \prod_{i=1}^{n-1} f(x_i,x_n) \cdot \varphi \circ U_{x_{n-1}} \cdots U_{x_1} D_{x_n} \la 
  \\
 = \cdots 
 = \prod_{i=1}^n g(x_i) \prod_{1 \leq i<j \leq n} f(x_i,y_j) \cdot
 \varphi \circ D_{x_{1}} \cdots D_{x_n} \la 
 \\
= \prod_{i=1}^n g(x_i) \prod_{1 \leq i<j \leq n} f(x_i,y_j) \cdot
\sum_{\mu} \varphi(\mu) \left\langle D_{x_{1}} \cdots D_{x_n} \la  , \mu \right\rangle
\\
 = \prod_{i=1}^n g(x_i) \prod_{1 \leq i<j \leq n} f(x_i,y_j)
 \sum_\nu \varphi(\nu) s_{\la/\nu}^\perp(\x). \qedhere
\end{multline*}
\end{proof}

\section*{Acknowledgement}
The author thanks Seamus Albion, Gabriel Frieden, Ilse Fischer and Christian Krattenthaler for helpful discussions.

%
%
%
%

\bibliographystyle{abbrvurl}
\bibliography{Littlewood_revisited}

\end{document}